\documentclass[11pt]{article}

\def\UseSection{%%
        \numberwithin{equation}{section}
        \newtheorem{theorem}    {Theorem}[section]
        \DefineTheorems % Use this to define other environments to be
}
\usepackage[textwidth=500pt,textheight=650pt,centering]{geometry} % 11pt

\usepackage{amsfonts}
\usepackage{amsmath,amssymb,amsthm}
\usepackage{appendix}
\usepackage{bbm} % used in \newcommand{\1}{\mathbbm{1}}
\usepackage{amsbsy}
\usepackage{enumerate}
\usepackage{cite}
\usepackage{comment}

\usepackage[bookmarks, colorlinks, breaklinks,
pdftitle={},pdfauthor={}]{hyperref}

\hypersetup{
  linkcolor=black,
  citecolor=black,
  filecolor=black,
  urlcolor=black}

\usepackage{verbatim}
\usepackage{tikz}
\usepackage{stmaryrd}
\usepackage{color}

\newcommand{\magenta}{\color{magenta}}

\newcommand{\black}{\black}

\usepackage[font=small,labelfont=bf]{caption}

\numberwithin{equation}{section}

\usetikzlibrary{calc,decorations.markings}
\usetikzlibrary{decorations.pathmorphing}
\usetikzlibrary{decorations.pathreplacing}
\usetikzlibrary{patterns}
\usetikzlibrary{arrows}

\newcommand{\bb}[1]{\mathbb{#1}}

\newcommand{\1}{\mathbbm{1}}

\newcommand{\D}[0]{{\, \operatorname{d}}}

\newcommand{\blank}[1]{}

\newcommand{\E}{\bb E}
\newcommand{\R}{\bb R}
\newcommand{\Z}{\bb Z}
\newcommand{\C}{\bb C}

\newcommand{\T}{\bb T}
\renewcommand{\P}{\bb P}

\newcommand{\Ccal}   {\mathcal{C}}

\newcommand{\Gcal}   {\mathcal{G}}

\newcommand{\Lcal}   {\mathcal{L}}
\newcommand{\Mcal}   {\mathcal{M}}

\newcommand{\Ucal}   {\mathcal{U}}

\newcommand{\Wcal}   {\mathcal{W}}

\newcommand{\veee}[1]{\langle #1 \rangle}
\newcommand{\xvee}{\veee{x}}

\newcommand{\yvee}{\veee{y}}

\newcommand{\xyvee}{\veee{x-y}}

\newcommand{\chem}[1]{{\magenta #1}}

\newcommand{\nnb}	{\nonumber \\}
\newcommand{\bubble}{{\sf B}}

\def\DefineTheorems{%%
	\newtheorem{lemma}      [theorem] {Lemma}
	
	\newtheorem{conj}        [theorem] {Conjecture}
	
	\newtheorem{prop}        [theorem] {Proposition}
	\theoremstyle{definition}% ``defn'' theorem style
	\newtheorem{defn}       [theorem] {Definition}
	\newtheorem{rk}       [theorem] {Remark}
}

\UseSection
\newcommand{\lbeq}[1]  {\label{e:#1}}
\newcommand{\refeq}[1] {\eqref{e:#1}}    % AMS-LaTeX trick!

\title{Weakly self-avoiding walk on a high-dimensional torus}

 \author{
    Emmanuel Michta\thanks{Department of Mathematics,
     University of British Columbia,
     Vancouver, BC, Canada V6T 1Z2.
     Michta: \url{https://orcid.org/0000-0001-7222-0422}, {\tt michta@math.ubc.ca}.
     Slade:  \url{https://orcid.org/0000-0001-9389-9497}, {\tt slade@math.ubc.ca}.}
    \and
   Gordon Slade$^*$}

 \date{\vspace{-5ex}} %for arXiv

\begin{document}
\maketitle

\begin{abstract}
How long does a self-avoiding walk on a discrete $d$-dimensional
torus have to be before it begins
to behave differently from a self-avoiding walk on $\Z^d$?
We consider a version of this question for weakly self-avoiding walk on
a torus in dimensions $d>4$.  On $\Z^d$ for $d>4$, the partition function for $n$-step weakly
self-avoiding walk is known to be
asymptotically purely exponential, of the form $A\mu^n$,
where $\mu$ is the growth constant for weakly self-avoiding walk on $\Z^d$.
We prove
the identical asymptotic behaviour $A\mu^n$
on the torus (with the same $A$ and $\mu$ as on $\Z^d$)
until $n$ reaches order $V^{1/2}$,
where $V$ is the number of vertices in the torus.  This shows that the walk must have
length of order at least $V^{1/2}$ before it ``feels'' the torus in its leading
asymptotics.
Our results support the conjecture that
the behaviour of the partition function does change once $n$ reaches $V^{1/2}$, and we
relate this to a conjectural
critical scaling window which separates the dilute phase $n \ll V^{1/2}$ from the dense phase $n \gg V^{1/2}$.  To prove the conjecture and to establish the
existence of the scaling window remains a challenging open problem.
The proof uses a novel lace expansion analysis based on the ``plateau'' for
the torus two-point function obtained in previous work.
\end{abstract}

%\noindent
%Keywords: self-avoiding walk; lace expansion; torus plateau; phase transition.
%
%\medskip \noindent
%MSC2010 Classifications: 60K35, 82B27, 82B41.

\section{Introduction}
\label{sec:1}

\subsection{Weakly self-avoiding walk}

An $n$-step walk on $\Z^d$ is a function $\omega :\{0,1,\ldots,n\} \to \Z^d$ with
$\|\omega(i)-\omega(i-1)\|_1=1$ for $1 \le i \le n$.
For $x \in \Z^d$,
let $\Wcal_n(x)$ denote the set of $n$-step walks with $\omega(0)=0$ and $\omega(n)=x$.
For an $n$-step walk $\omega$, and for $0 \le s < t \le n$, we define
\begin{equation}
\label{e:Ustdef}
    U_{st}(\omega) =
    \begin{cases}
        -1 & (\omega(s)=\omega(t))
        \\
        0 & (\omega(s)\neq \omega(t)).
    \end{cases}
\end{equation}
Given $\beta \in [0,1]$
and $x \in \Z^d$, we define the \emph{partition function} $c_n$ by
\begin{equation}
\label{e:cndef}
    c_n(x) = \sum_{\omega\in \Wcal_n(x)} \prod_{0\le s<t\le n}(1+\beta U_{st}(\omega)),
    \qquad
    c_n = \sum_{x\in\Z^d}c_n(x).
\end{equation}
We typically omit the dependency on $\beta$ in our notation.
The product in \eqref{e:cndef}
discounts $\omega$ by a factor
$1-\beta$  for each pair $s<t$
with an intersection $\omega(s)=\omega(t)$.
For $\beta=0$, $c_n$ is simply the number of $n$-step walks, which is $(2d)^n$.
For $\beta=1$, $c_n$ is the number of $n$-step (strictly) self-avoiding walks.
The case $\beta \in (0,1)$ is the \emph{weakly self-avoiding walk}.
All known theorems are consistent with the hypothesis of universality,
which asserts in particular
that the asymptotic behaviour of the weakly self-avoiding walk is the
same as that of strictly self-avoiding walk.  Since the presence of a small parameter
$\beta$ can be helpful, the weakly self-avoiding walk has often been studied, as in
\cite{BS85} where Brydges and Spencer invented their lace expansion.
The weakly self-avoiding walk is sometimes called the Domb--Joyce model, after \cite{DJ72}.

We are primarily interested in walks not on $\Z^d$, but rather on a discrete torus
$\T_r^d = (\Z/r\Z)^d$ for an integer $r \ge 3$ which will be taken to be large.
The \emph{volume} of the torus is $V=r^d$.
Torus walks are defined as on $\Z^d$ but with the difference $\omega(i)-\omega(i-1)$
computed using addition modulo $r$ in each component.
For $x \in \T_r^d$, we let $\omega\in \Wcal_n^\T(x)$ denote the set of $n$-steps torus walks from $0$ to $x$.
In the same way as on $\Z^d$, we define torus quantities
\begin{equation}
\label{e:cnTdef}
    c_n^\T(x) = \sum_{\omega\in \Wcal_n^\T(x)} \prod_{0\le s<t\le n}(1+\beta U_{st}(\omega)),
    \qquad
    c_n^\T = \sum_{x\in\T_r^d}c_n^\T(x).
\end{equation}

Both sequences $c_n$ and $c_N^\T$ are submultiplicative, e.g., $c_{n+m} \le c_nc_m$,
and therefore (see \cite[Section~1.2]{MS93})
the limits $\mu=\lim_{n\to \infty}c_n^{1/n}$ and $\mu^\T =
\lim_{n\to \infty}(c_n^\T)^{1/n}$ exist, with $c_n \ge \mu^n$ and $c_n^\T \ge (\mu^\T)^n$
for all $n \ge 0$.
However there is a very big difference:  unsurprisingly $\mu$ is finite and strictly positive
(in fact in the interval $[\mu_{1},2d]$ where $\mu_1=\mu(\beta=1)$ is the \emph{connective
constant}),
but $\mu^\T$ is
equal to zero when $\beta \in (0,1]$.
This can be seen as follows.  Given a torus walk $\omega$ and a point $x$ in
the torus,
let $\ell_x(\omega)$ be the local time at $x$, i.e., the number of times the walk visits $x$.
Then $\sum_{x \in \T_r^d}\ell_x(\omega) = n+1$,
and with the degenerate binomial coefficients $\binom{0}{2}$ and $\binom{1}{2}$ equal to zero,
\begin{align}
	\prod_{0\le s<t\le n}(1+\beta U_{st}(\omega))
	&= (1-\beta)^{\sum_{x \in \T_r^d}\binom{\ell_x(\omega)}{2}}
    = (1-\beta)^{\frac12 \sum_{x \in \T_r^d} \ell_x(\omega)^2-\frac12 (n+1)},
\end{align}
from which, via $1-\beta \le e^{-\beta}$ and the Cauchy--Schwarz inequality, we obtain
\begin{align}
\label{e:cnTrough}
	c_n^\T
	&\leq (2d)^n e^{-\frac12\beta( (n+1)^2/V-  (n+1))}.
\end{align}
The inequality \eqref{e:cnTrough} shows that $\mu^\T=0$ when $V$ is fixed and
$\beta >0$.

\subsection{Notation}

We write $f \sim g$ to mean $\lim f/g =1$, $f\prec g$ to mean $f \leq c_1 g$ with $c_1>0$ and $f \succ  g$ to mean  $g \prec  f$. We also write $f \asymp g$ when $ g \prec f \prec g$.
Constants are permitted to depend on $d$.

\subsection{Main result}

In order to compare walks in $\T_r^d$ and in $\Z^d$ we introduce \emph{the lift of a
torus walk} which is to be thought of as the unwrapping to $\Z^d$ of a walk on $\T_r^d$.
In detail, for any $x \in \T_r^d$ we let $x_r$ denote the unique representative
in $[-\frac r2, \frac r2)^d \cap \Z^d$
of the equivalence class of $x$.
Given a torus walk $\omega \in \Wcal_n^\T = \cup_{x\in \T_r^d}\Wcal_n(x)$,
we define the $\Z^d$ walk $\bar \omega \in \Wcal_n= \cup_{x\in \Z^d}\Wcal_n(x)$,
the \emph{lift} of $\omega$, by $\bar\omega(0)=0$ and
\begin{equation}
\label{e:liftdef}
	\bar \omega(k) = \bar\omega(k-1) + (\omega(k) - \omega(k-1))_r \qquad (1\leq k \leq n).
\end{equation}
Due the nearest neighbour constraint, when $r \ge 3$ this lift operation is a bijection from $\Wcal_n^\T$ onto $\Wcal_n$.

Since each walk on the torus lifts in a one-to-one manner
to a walk on $\Z^d$, with the lifted walk having no more self-intersections than
the original walk, in general $c_{n}^{\T}\le c_n$.
A walk of length less than $r$ has the same number of intersections as its lift,
so $c_{n}^{\T}=c_n$ for $n < r$.  Our interest is in how this equality for $n<r$
is asymptotically preserved as $n$ increases, and in what happens when the torus behaviour
departs from that of $\Z^d$.  We refer to the $V$-dependent regime of $n$ values for which
$c_{n}^{\T}$ and $c_n$ have the same leading asymptotic behaviour as the \emph{dilute phase}.

We address this question for dimensions $d>4$, where the $\Z^d$ theory is essentially
complete.  In particular, it is proved in \cite{HHS98}
that for $d>4$ with $\beta>0$ sufficiently small\footnote{Although not
explicit in \cite{HHS98},
the factor $\beta$ in the error term can be extracted from the proof.}
\begin{equation}
\label{e:cnasy}
    c_n = A\mu^n (1+O(\beta n^{-(d-4)/2})),
\end{equation}
with $A = 1+O(\beta)$.
Related results with different error estimates can be found in
\cite{BS85,GI92,KLMS95,HHS98,ABR16}.  For the case of strictly self-avoiding walk
($\beta =1$) in dimensions $d > 4$ it is proved in \cite{HS92a} that
\begin{equation}
\label{e:cneta}
    c_n=A_{1}\mu_{1}^n[1+O(n^{-\eta})]
\end{equation}
where the subscripts indicate $\beta=1$ and where
$\eta$ is any positive number such
that\footnote{Presumably \eqref{e:cneta} remains true with $\eta$ replaced by $\frac{d-4}{2}$ as in
\eqref{e:cnasy} but this has not been proved.}
\begin{equation}
\label{e:etadef}
    \eta < \min\{\textstyle{\frac{d-4}{2}} , 1\}
    .
\end{equation}

In view of the above simple observation that $c_n^\T=c_n$ for $n<r$,
if $n \to \infty$ and $r \to \infty$ in such a manner that
$n < r$ then the torus walk behaves exactly as the $\Z^d$ walk, namely $c_{n}^{\T}\sim A\mu^n$.
In any
case, $c_{n}^{\T} \le c_n \sim A\mu^n$.
Our main result is the following theorem, which shows that the weakly self-avoiding walk
does not feel the effect of the torus in its leading asymptotic
behaviour at least
until the walk's length reaches $V^{1/2}$.
In the statement of Theorem~\ref{thm:dilute}, $A$ and $\mu$ have the same values as in the formula \eqref{e:cnasy} for $\Z^d$.
A related theorem for the strictly self-avoiding walk on the hypercube
is proven in \cite{Slad22}.

\begin{theorem}
\label{thm:dilute}
For $d>4$ and $C_0>0$, if $\beta >0$ is sufficiently small and
if
$n \le  C_0 V^{1/2}$, then
\begin{equation}
\label{e:cnmr}
    c_{n}^{\T} = A\mu^n
    \left[ 1+O(\beta)\Big( \frac{1}{n^{(d-4)/2}} + \frac{n^2}{V}\Big) \right]
    ,
\end{equation}
where the constant in the error term depends on $C_0$ but not on $n,V,\beta$.
\end{theorem}

For example, if $n=V^p$ with $p\leq \frac 12$ then \eqref{e:cnmr} asserts that
\begin{equation}
    c_{n}^{\T} = A\mu^n
    \left[ 1+O(\beta)\Big( \frac{1}{n^{(d-4)/2}} + \frac{1}{n^{(1-2p)/p}}\Big) \right]
    \qquad (n=V^p).
\end{equation}
The $n^{-(d-4)/2}$ term is dominant when
\begin{equation}
    \frac{1-2p}{p} > \frac{d-4}{2}
\end{equation}
which happens when $p< \frac 2d$, i.e., when $n$ grows more slowly than $r^2$.
As $p$ increases from $\frac 2d$ to $\frac 12$, the error term $n^{-(1-2p)/p}$
becomes increasingly significant and stops decaying at $p=\frac 12$.
We regard this as a reflection of the
conjectural failure of $A\mu^n$ to accurately capture
the asymptotic behaviour of $c_n^\T$ once $n$ reaches $V^{1/2}$
(see Section~\ref{sec:window}
for further discussion of this point).  If the $n^{-(1-2p)/p}$ term is indeed sharp,
then the effect of the torus enters into the error term already when $p=\frac 2d$, though
not yet in the leading term.  However, note that
Theorem~\ref{thm:dilute} does
address the behaviour of $c_n^\T$ when $n$ is
closer to $V^{1/2}$ than $V^p$ with $p< \frac 12$, for example $n = V^{1/2}/(\log V)$.

Finally we remark that we are investigating a regime for which it is not possible
to see the typical distance travelled after $n$ steps being of order $n^{1/2}$
(as it is for self-avoiding walk on $\Z^d$ for $d>4$ \cite{BS85,HS92a}).
For example,
for $n=V^p$ we have $n^{1/2} = r^{pd/2}$ which for $p>\frac 2d$ is much greater
than the maximal torus displacement.  Such diffusive behaviour on the torus would
instead be visible as a Brownian scaling limit on a continuum torus.
This phenomenon is studied in detail in \cite{Mich22}.
It would also be of interest to attempt to prove an analogue of
Theorem~\ref{thm:dilute} for the strictly self-avoiding walk ($\beta=1$)
in dimensions $d >4$, but our method relies on the control of the near-critical
two-point function from \cite{Slad20_wsaw} and this is currently available only
for small $\beta$.

\subsection{Generating functions}

We define the \emph{susceptibility} for $\Z^d$ and for the torus $\T_r^d$,
for complex $z\in \C$, by
\begin{equation}
\label{e:chidefT}
    \chi(z)
    = \sum_{n=0}^\infty c_n z^n,
    \qquad
    \chi^\T(z)
    = \sum_{n=0}^\infty c_n^\T z^n.
\end{equation}
The radius of convergence for $\chi$ is $z_c=\mu^{-1}$, whereas $\chi^\T$ is entire
when $\beta \in (0,1]$ by \eqref{e:cnTrough}.
As $z \to z_c$ the susceptibility on $\Z^d$ diverges linearly (this is well-known
and we provide a proof in Section~\ref{sec:ELpf})
\begin{equation}
\label{e:chiZdasy}
    \chi(z)  = \frac{A}{1- z/z_c}\left( 1 + \frac{O(\beta)}{(1-z/z_c)^{(d-4)/2}} \right) ,
\end{equation}
consistent with \eqref{e:cnasy}.
For $z\in \C$ with $|z|< z_c$
we also define the generating function $H$ and its coefficients $h_n$ by
\begin{equation}
    H(z) = \chi^\T(z) - \chi(z) = \sum_{n=0}^\infty h_n z^n.
\end{equation}
By definition and by \eqref{e:cnasy},
\begin{equation}
    c_n^\T = c_n + h_n = A\mu^n \left[ 1 + O(\beta n^{-(d-4)/2}) + A^{-1} h_n z_c^{n} \right],
\end{equation}
so it suffices to prove that
\begin{equation}
\label{e:hmu}
    h_n z_c^{n} = O(\beta n^2/V)+ O(\beta n^{-(d-4)/2})
\end{equation}
when $n \le C_0 V^{1/2}$.

Let
\begin{equation}
    \zeta = z_c(1-V^{-1/2}), \qquad U = \{ z\in \C : |z| < \zeta\}.
\end{equation}
The choice $\zeta = z_c(1-V^{-1/2})$ is inspired by the conjectural form of the \emph{scaling window} which is discussed in detail in Section~\ref{sec:window}: $\zeta$ is a typical
point in the dilute side of the scaling window.  This choice is dual to our restriction
to walks whose length $n$ is at most $O(V^{1/2})$.
We will prove the following theorem.

\begin{theorem}
\label{thm:susceptibility}
Let $d>4$ and let $\beta>0$ be sufficiently small.
Then
\begin{equation}
\label{e:Hbd}
    |H(z)| \prec \frac{\beta }{V} \frac{1}{|1-z/\zeta|^{2}(1-|z|/\zeta)}
    + \frac{\beta }{V} \frac{r^2}{|1-z/\zeta|^{2}}
    \qquad (z\in U).
\end{equation}
\end{theorem}

Theorem~\ref{thm:susceptibility} is useful in combination with the
Tauberian theorem stated in the next lemma.
The lemma makes precise the intuition that if a power
series $f(z) = \sum_{n=0}^\infty a_n z^n$ has radius of convergence
$R>0$ and if $|f(z)|$ is bounded above by a multiple of $|R-z|^{-b}$
on the disk of radius $R$, with $b \geq 1$,
then $a_n$ should be not much worse than order $R^{-n}n^{b-1}$.
For a proof of the lemma
 see
\cite[Lemma~3.2(i)]{DS98}.

\begin{lemma}
\label{lem:Tauberian}
Suppose that the power series $f(z)=\sum_{n=0}^\infty a_n z^n$ has radius of convergence
at least $R>0$,
and that there exist $b>1$ and $c\ge 0$ such that
\begin{equation}
    |f(z)| \le K_1 \frac{1}{|1-z/R|^b (1-|z|/R)^c}
\end{equation}
for all $z \in \C$ with $|z|<R$.
Then $|a_n| \le K_1 K_2 n^{b+c-1}R^{-n}$ with $K_2$ depending only on $b$, $c$.
\end{lemma}

We now show that Theorem~\ref{thm:dilute} is an
immediate corollary of Theorem~\ref{thm:susceptibility}.

\begin{proof}[Proof of Theorem~\ref{thm:dilute}]
The combination of Theorem~\ref{thm:susceptibility} with Lemma~\ref{lem:Tauberian}
(with $f=H$, $R=\zeta$)
immediately gives
\begin{equation}
    |h_n| \prec
    \beta \Big(\frac{n^2}{V} + \frac{nr^2}{V}\Big)\zeta^{-n}
    .
\end{equation}
Therefore, by definition of $\zeta$,
\begin{equation}
    |h_n|z_c^n \prec  \beta\Big(\frac{ n^2}{V}+\frac{nr^2}{V}\Big)  \frac{1}{(1-V^{-1/2})^n}.
\end{equation}
Since $n \le C_0 V^{1/2}$ by assumption, the above right-hand side is bounded by
a $C_0$-dependent constant multiple of
$\beta V^{-1}(n^2+nr^2)$.
Finally, the error term
$\beta nr^2/V$ can be absorbed into one of the two error terms of \eqref{e:hmu} since
\begin{equation}
    \frac{nr^2}{V}
    \le
    \begin{cases}
        \frac{n^2}{V} & (n \ge r^2)
        \\
        \frac{n}{r^{d-2}} \le \frac{n}{n^{(d-2)/2}} = \frac{1}{n^{(d-4)/2}} & (r^2 \ge n).
    \end{cases}
\end{equation}
This proves \eqref{e:hmu} and therefore completes the proof.
\end{proof}

It remains to prove Theorem~\ref{thm:susceptibility}.
The proof is based on the lace expansion
for self-avoiding walk \cite{BS85} (see
\cite{MS93,Slad06} for introductions), both on $\Z^d$ and on the torus.
Previously the lace expansion has been applied to self-avoiding
walk only on the infinite lattice.
The proof also relies in an important way
on the decay of the near-critical $\Z^d$ two-point function, and
the ``plateau'' for the torus two-point function implied by that decay, which are obtained
in \cite{Slad20_wsaw} and which we discuss in detail in Section~\ref{sec:plateau}.

Before turning to the proof of Theorem~\ref{thm:susceptibility}, we complete
Section~\ref{sec:1} with further context for Theorem~\ref{thm:dilute} as well as conjectures
concerning potential extensions.

\subsection{Susceptibility and expected length}
\label{sec:chiEL}

\subsubsection{Susceptibility}

It is a trivial consequence of the inequality $c_n^\T \le c_n$ that in all dimensions
and for all $\beta \in [0,1]$ the susceptibilities obey the inequality $\chi^\T(z) \le \chi(z)$
for all $z \in [0,z_c]$.  For $d >4$ and sufficiently small $\beta>0$ a complementary
lower bound is proved in \cite{Slad20_wsaw}, as we discuss further in Section~\ref{sec:plateau}.
Together with the general upper bound, the lower bound states that there is a constant $c_1$ such that
\begin{equation}
\label{e:chiupperlower}
    c_1 \chi(z) \le \chi^\T(z) \le \chi(z),
\end{equation}
with the lower bound valid for any real
$z \in  [ z_c- c_3 V^{-2/d}, z_c - c_4\beta^{1/2} V^{-1/2} ]$
for sufficiently large constants $c_3,c_4 >0$ (independent of $\beta$).
We will prove the following extension of \eqref{e:chiupperlower}.

\begin{theorem}
\label{thm:chi_torus_good}
Let $d >4$ and let $\beta$ be sufficiently small.
For any $z\in \C$ with $|z|\le z_c(1-V^{-1/2})$,
\begin{align}
\lbeq{chichiT}
    \chi^\T(z) & =
    \chi(z) \left( 1 + O\Big(\beta \frac{\chi(z)(r^2 +\chi(|z|))}{V} \Big) \right).
\end{align}
\end{theorem}

In particular, since $|\chi(z)| \le \chi(|z|)\le O(V^{1/2})$
when $|z|\le z_c(1-V^{-1/2})$, the error term in \refeq{chichiT}
is bounded by $O(\beta)$ uniformly in $|z|\le z_c(1-V^{-1/2})$ and hence in this disk we have
\begin{align}
    \chi^\T(z) & =
    \chi(z) \left( 1 + O(\beta ) \right).
\end{align}
Theorem~\ref{thm:chi_torus_good} can be extended to $|z|\le z_c(1-s V^{-1/2})$
for any small positive $s$ at the cost of taking $\beta$ to be small depending on
$s$, but we do not emphasise this option in the statement of the theorem.  We discuss this
further in Remark~\ref{rk:beta-s}.

\subsubsection{Expected length}
\label{sec:EL}

It is common in the analysis of self-avoiding walk on a finite graph to introduce
a probability measure on variable length walks.
In our torus context, the set of such walks is
$\cup_{x\in \T_r^d}\cup_{n=0}^\infty \Wcal^\T_n(x)$,
and the measure is defined for $z \ge 0$ by
\begin{equation}
\label{e:PzT}
    \P_z^\T(\omega)
    = \frac{1}{\chi^\T(z)} z^{|\omega|}\prod_{0\le s<t\le |\omega|}(1+\beta U_{st}(\omega))
\end{equation}
with $|\omega|$ equal to the number of steps of the walk $\omega$.
The \emph{length} is the discrete random variable $L$ whose probability mass function
is given, for $n \ge 0$, by
\begin{equation}
    \P_z^\T(L=n) = \frac{1}{\chi^\T(z)} c^\T_n z^n.
\end{equation}
The \emph{expected length} is therefore
\begin{equation}
    \E_z^\T L = \frac{z\partial_z \chi^\T(z)}{\chi^\T(z)}.
\end{equation}

\begin{theorem}
\label{thm:EL}
Let $d>4$ and let $\beta>0$ be sufficiently small.
Let $z\in [\frac 12 z_c,z_c(1-V^{-1/2})]$.  Then
\begin{align}
	\E_{z}^{\T}L
    &=
    \frac{z}{z_c}\frac{1}{1-z/z_c}
    \left( 1 + O\Big(\frac{\beta}{V(1-z/z_c)^2}\Big)
    + O\Big(\frac{\beta  r^2}{V(1-z/z_c)}\Big)
    + O(\beta(1-z/z_c)^{1-\frac12(d-6)_+} )
    \right) ,
\end{align}
where in the exponent $x_+=\max\{0,x\}$ and the third error term is
replaced by $O(\beta (1-z/z_c)|\log(1-z/z_c)|)$ when $d=6$.
\end{theorem}

In particular, if $z=z_c/(1+V^{-p})$ with $p< \frac 12$ then Theorem~\ref{thm:EL} gives
\begin{align}
	\E_{z_c/(1+V^{-p})}^{\T}L
    &  \sim
    V^p
\end{align}
(with an explicit error term).
For $z=z_c/(1+sV^{-1/2})$
with $s\ge 1$, it instead gives
\begin{align}
	\E_{z_c/(1+sV^{-1/2})}^{\T}L
    &=
    \frac{1}{s}V^{1/2}
    \left( 1 + O(\beta) \right) .
\end{align}
We conjecture that for any $\beta \in (0,1]$ and any real $s$ (positive or negative or zero) there exist constants $c_s,C_s$ such that
\begin{equation}
\label{e:EL}
    c_s V^{1/2} \le \E_{z_c(1-sV^{-1/2})}^{\T}L \le C_s V^{1/2}.
\end{equation}
However our results do not go beyond the range $s \ge 1$ stated in Theorem~\ref{thm:EL},
in the sense that in our proof decreasing $s$ requires that $\beta$ also
be decreased, as discussed in Remark~\ref{rk:beta-s}.

\subsection{Self-avoiding walk on the complete graph}
\label{sec:Kn}

Statistical mechanical models above their upper critical dimension often have
the same critical behaviour that occurs when the model is formulated on the complete
graph.  A classic example of this is the Ising model in dimensions $d>4$
and its complete graph version known
as the Curie--Weiss model.  So it is natural to compare Theorems~\ref{thm:dilute}
and \ref{thm:EL} with
the behaviour of self-avoiding walk on the complete graph.

The number of $n$-step (strictly) self-avoiding walks on the complete graph on $V$ vertices is
simply $c_n^{\mathbb{K}} = \frac{v!}{(v-n)!}$, where $v=V-1$.
In the limit in which $v \to \infty$, and assuming for simplicity
that $n=o(v^{2/3})$ (this implies in particular that $v-n\to \infty$), a
calculation using Stirling's formula shows that
\begin{equation}
\label{e:cnKn}
    c_n^{\mathbb K}
     = \frac{v!}{(v-n)!} = v^n e^{-n^2/2v} [1+o(1)].
\end{equation}
This shows that on the complete graph the
leading asymptotic behaviour of the
number of $n$-step self-avoiding walks is purely exponential $v^n$ as long as $n \ll V^{1/2}$,
and that a finite-volume correction occurs once $n$ reaches and exceeds $V^{1/2}$.
This could be called the ``birthday effect'' since $c_n^{\mathbb K}$
is the number of different ways
that $n$ individuals can have distinct birthdays  when $v=365$.

Despite its simplicity,
a thorough analysis of the critical behaviour of self-avoiding walk on the complete graph
has only recently been given in \cite{DGGNZ19,Slad20}.
In \cite[Section~1.5]{Slad20}, it is proposed that a correct definition of
a critical point for self-avoiding
walk on a high-dimensional transitive graph $\mathbb{G}$ with $V$ vertices
is the value of $z_c^{\mathbb{G}}$ for which the
susceptibility (the generating function for $c_n$ on the graph $\mathbb{G}$) equals $\lambda V^{1/2}$,
for any fixed $\lambda >0$, and that the critical scaling window consists of those $z$ for
which
$|z-z_c^{\mathbb{G}}| \le O(z_c^{\mathbb{G}} V^{-1/2})$.
While this definition remains a reasonable one for the torus $\T_r^d$ as an
\emph{intrinsic} definition of the critical point, there is also the \emph{extrinsic}
candidate $z_c=\mu^{-1}$ which is the critical point for $\Z^d$.  By making use of
the essentially complete understanding of the weakly self-avoiding walk on $\Z^d$ for
$d>4$, we are able to work in this paper without
explicitly using the intrinsically defined critical value.

On the complete graph,
the susceptibility $\chi^{\mathbb{K}}(z) = \sum_{n=0}^\infty c_n^{\mathbb K} z^n$
can be written exactly in terms of the incomplete Gamma function.  This allows
for an explicit description of the phase transition in \cite[Theorem~1.1]{Slad20}, as follows.
Let $\alpha> \frac 12$.  If $z_V \le V^{-1}(1-V^{-\alpha})$ then
$\chi^\mathbb{K}(z_V) \sim (1-Vz_V)^{-1}$.  This is the dilute phase.
The critical window consists of values of the
form $z_V= V^{-1}(1+ s V^{-1/2})$ with $s \in \R$, and in this case
$\chi^\mathbb{K}(z_V)$ is asymptotically an explicit $s$-dependent multiple of $V^{1/2}$.
In the dense phase, which includes $z_V \ge V^{-1}(1+V^{\alpha-1})$, the susceptibility $\chi^\mathbb{K}(z_V)$
grows exponentially in $V$.
Similarly, on the complete graph the \emph{expected length}
makes a transition from bounded when deep in the dilute phase, to order $V^{1/2}$ in the critical window,
to order $V$ deep in the dense phase \cite{DGGNZ19,Slad20}.

\subsection{The scaling window, the dense phase, and boundary conditions}
\label{sec:window}

The complete graph provides a guide for what can be expected for self-avoiding walk
on other high-dimensional graphs (without boundary).  In particular, it
can be expected that for the torus $\T_r^d$ with $d>4$
there is a scaling window of the form $z=z_c(1+sV^{-1/2})$ for
$s \in \R$, within which the susceptibility and expected length both scale like $V^{1/2}$,
below which both remain bounded (the dilute phase), and above which the susceptibility and expected length
are respectively exponentially large and of order $V$ (the dense phase).  In this section, we elaborate on this picture.

\subsubsection{Scaling window and dense phase}

By analogy with the theory of self-avoiding walk on the complete graph
discussed in Section~\ref{sec:Kn}, we are led to the
conjecture that, for weakly or strictly self-avoiding walk
on the torus $\T_r^d$ in dimensions $d>4$, the interval $z \in (0,\infty)$ is
divided into three regimes.  For $z$ written as $z=z_c(1+\epsilon)$
with $\epsilon \in (-1,\infty)$, these regimes are:
\begin{itemize}
\item the \emph{dilute phase} $\epsilon \ll - V^{-1/2}$:
\[
    \chi^{\T} \asymp \epsilon^{-1},
    \qquad
    c_n^{\T} \sim A\mu^n \;\; \text{for $n \ll V^{1/2}$},
    \qquad
    \E_z^{\T}L \asymp \epsilon^{-1} ;
\]
\item the \emph{critical window} $|\epsilon| \asymp V^{-1/2}$:
\[
     \chi^{\T} \asymp V^{1/2}, \qquad
    c_n^{\T} \asymp  \mu^n \;\; \text{for $n \asymp V^{1/2}$},
    \qquad
    \E_z^{\T}L \asymp V^{1/2};
\]
\item the \emph{dense phase} $\epsilon \gg V^{-1/2}$:
\[
    \chi^{\T}  \;\;\text{exponential in $V$},
    \qquad
    c_n^{\T} \ll \mu^n \;\;\text{for $n \gg V^{1/2}$},
    \qquad
    \E_z^{\T}L \asymp V \frac{\epsilon}{1+\epsilon}
    .
\]
\end{itemize}
Theorems~\ref{thm:dilute} and \ref{thm:chi_torus_good}--\ref{thm:EL}
are consistent with the above
when $\epsilon \le -V^{-1/2}$.
Closely related results have been obtained for self-avoiding walk on the hypercube
\cite{Slad22}.
The remainder of the scaling window, for which $\epsilon =sV^{-1/2}$ with $s>-1$,
as well as the dense phase, for which $\epsilon \gg  V^{-1/2}$,
remain open for self-avoiding walk on the torus.
Investigations of different questions concerning the dense phase of self-avoiding walk can
be found in \cite{BGJ05,DKY14,GI95,GJ22,Yadi16}.

A quantitative conjecture for the susceptibility in the scaling window, which adds considerable
precision to the scaling window, has very recently been made in \cite{MPS23}
on the basis of a renormalisation group analysis of the
$4$-dimensional $n$-component hierarchical $|\varphi|^4$ model.\footnote{$d=4$ involves
corrections both to the window size and susceptibility which are logarithmic in the volume.}
A related conjecture for the expected length appeared earlier in
\cite{DGGZ22} along with supporting numerical results.
To state the conjecture of \cite{MPS23},
for $s \in \R$ we define\footnote{The function $I_1(s)$ can be rewritten in
terms of the profile for
the susceptibility on the complete graph appearing in \cite[Theorem~1.1]{Slad20}
via
\[
 I_1(s)
= e^{s^2/4 } \int_0^{\infty} e^{-(y + s / 2)^2} dy
= \frac{\sqrt{\pi}}{2} e^{s^2/4} \operatorname{erfc} (s / 2) .
\]}
\begin{equation}
    I_1(s)=  \int_0^\infty t\, e^{-\frac 14 t^4 - \frac 12 st^2} dt.
\end{equation}

\begin{conj}
\label{conj:I1}
For $d>4$ and all $\beta\in (0,1]$, there are constants
$\lambda_1,\lambda_2$ depending on $d$ and $\beta$ (but not on $s,V$) such that,
for all $s \in \R$,
 the universal profile for the
susceptibility in the scaling window
is given by
\begin{equation}
    \lim_{V \to \infty} V^{-1/2}\chi^{\T}(z_c(1+\lambda_1 s V^{-1/2}))= \lambda_2 I_1(s).
\end{equation}
Also, for $z=z_c(1+\lambda_1 sV^{-1/2})$, the rescaled length $V^{-1/2}L$ converges in distribution to a random variable $X_s$ whose distribution is
that of $-s+Z_s$ where $Z_s$ is a standard normal random variable
conditioned to exceed $s$.  Explicitly, the random variable $X_s$ has moment generating function
$\E e^{tX_s} =I_1(s+t)/I_1(s)$).
\end{conj}

The scaling window discussed above
parallels the established theory for bond percolation
on a high-dimensional torus ($d>6$)
\cite{BCHSS05a,HH17book} for which the critical susceptibility and scaling window width
instead respectively involve
powers $V^{1/3}$ and $V^{-1/3}$.
In particular, the $\Z^d$ critical point lies in
the scaling window for a high-dimensional torus; this was proved in \cite{HHII11} and
very recently an alternate proof based on the plateau for the torus two-point function
was given in \cite{HMS22}.

\subsubsection{Conjectured bound on $c_n^\T$}

Theorem~\ref{thm:dilute} does not address the question of how $c_n^\T$ behaves for $n$ of order
larger than $V^{1/2}$.
It is natural to conjecture that, as in \eqref{e:cnKn},
the following bound holds  for all $n \geq 0$,  both for weakly and strictly self-avoiding walk.

\begin{conj}
\label{conj:cn}
Let $d>4$ and let $\beta\in (0,1]$.
There exist $K,\alpha>0$ (depending on $\beta$) such that, for all $n\ge 0$,
\begin{equation}
\label{e:conjcn}
    c_{n}^{\T} \le K \mu^n e^{- \alpha n^2/V}.
\end{equation}
\end{conj}

Conjecture~\ref{conj:cn} is consistent with Theorem~\ref{thm:dilute} in the
sense that
for $n\le V^{1/2}$ the factor $e^{- \alpha n^2/V}$ reproduces the $n^2/V$
error term in \eqref{e:cnmr}.
If Conjecture~\ref{conj:cn} is correct, then the effect of the finite torus
would be seen once $n$ reaches order $V^{1/2}$, as the
upper bound on $c_{n}^{\T}$ would then grow more slowly than
$A\mu^n$.
The $\beta$-dependence of  $\alpha$ is natural, since for $\beta=0$ we have
$c_n=c_n^\T=(2d)^n$ and there can be no diminution on the torus.

It is interesting to ask what improved bound on the susceptibility would be sufficient
in order to make progress on Conjecture~\ref{conj:cn}.
As we discuss next,
the following conjecture, which is much weaker than the very precise statement of
Conjecture~\ref{conj:I1}, implies an improvement to Theorem~\ref{thm:dilute}.

\begin{conj}
\label{conj:window}
Let $d>4$ and let $\beta>0$ be sufficiently small.  There exist $s>0$
and a constant $K_s$ such that
\begin{equation}
\label{e:windowconj}
    \chi^\T(z_c(1+sV^{-1/2})) \le K_s V^{1/2}.
\end{equation}
\end{conj}

Conjecture~\ref{conj:window} is implied by Conjecture~\ref{conj:cn}, since
the latter implies that for any $s>0$ we have
\begin{align}
    \chi^\T(z_c(1+sV^{-1/2}))
    & \le
    K
    \sum_{n=0}^\infty e^{-\alpha n^2/V} e^{sn/V^{1/2}}
    \nnb &
    \prec
    \int_0^\infty e^{-\alpha t^2/V + s t/V^{1/2}} \D t \asymp_s V^{1/2},
\end{align}
where $\asymp_s$ indicates that the implicit constants may depend on $s$.
Conversely, Conjecture~\ref{conj:window} implies a weaker form of
Conjecture~\ref{conj:cn} via the following very elementary but consequential lemma
 due to Hutchcroft \cite[Lemma~3.4]{Hutc19}.

\begin{lemma}
\label{lem:TauberianTom}
For $n \ge 0$ let $a_n \ge 0$ be submultiplicative, i.e., $a_{n+m} \le a_n a_m$ for all $m,n$.
Let $A(z) = \sum_{n=0}^\infty a_nz^n$.  Then, for every $n \ge 1$ and every $z \ge w >0$,
\begin{equation}
    a_n \le \frac{z^n}{w^{2n}} \left( \frac{A(w)}{n+1} \right)^2.
\end{equation}
\end{lemma}

When applied to the submultiplicative sequence $a_n=c_{n}^{\T}$
with $A = \chi^\T$, $z=z_c(1+sV^{-1/2})$ with $s>0$, and $w= \frac{n}{n+1}z$,
and assuming Conjecture~\ref{conj:window},
Lemma~\ref{lem:TauberianTom} implies the existence of a positive constant $K_s'$ (depending on $\beta$) such that
\begin{align}
\label{e:TomTauberian}
    c_{n}^\T &
    \le K_s' \frac{1}{z_c^n (1+sV^{-1/2})^n}
    \le K_s' \mu^n e^{-\frac 12 s nV^{-1/2}}
    ,
\end{align}
which involves a weaker exponentially decaying factor than in Conjecture~\ref{conj:cn}.
To verify the first inequality of \eqref{e:TomTauberian}, we used
$\chi^\T(w) \le \chi(w) \prec (z_c-w)^{-1}$ for $n \le (2s)^{-1}V^{1/2}$ and for larger $n$
we applied \eqref{e:windowconj}.

\subsubsection{Effect of boundary conditions}

A wider issue
concerns the effect of boundary conditions
on the critical behaviour of
statistical mechanical models in finite volume above the upper critical
dimension, and includes \cite{ZGFDG18,WY14,LM16} as a small sample of a much
larger literature.
For self-avoiding walk or the Ising model,
the emerging consensus is that with free boundary conditions
the susceptibility for a box of side $r$ is of order $r^2=V^{2/d}$ at the $\Z^d$ critical
point, whereas  with periodic boundary conditions (i.e., on the torus)
the susceptibility
behaves instead as $V^{1/2}$.  Theorem~\ref{thm:chi_torus_good}
and Conjecture~\ref{conj:window} concern the $V^{1/2}$ behaviour
for weakly self-avoiding walk in dimensions $d>4$.  A proof for free boundary conditions
remains open for self-avoiding walk models; numerical evidence is presented in
\cite{ZGFDG18}.
For the Ising model in dimensions $d>4$, with free boundary conditions
the $V^{2/d}$ behaviour
of the susceptibility has been proved
in \cite{CJN21}, and for periodic boundary conditions
 a  $V^{1/2}$ lower bound for the susceptibility is proved in \cite{Papa06}.
The Ising upper bound remains unproved on the torus.

\subsection{Enumeration of self-avoiding walks on $\Z^d$}

We are not aware of any previous rigorous work on
self-avoiding walk on a torus, but the enumeration of (strictly)
self-avoiding walks on infinite graphs, especially $\Z^d$,
has a long history going back to the middle
of the twentieth century.  Let $s_n$ denote the number of $n$-step strictly self-avoiding
walk on $\Z^d$ starting from the origin, and let $\mu_1 = \lim_{n\to\infty}s_n^{1/n}$
be the ($d$-dependent) \emph{connective
constant}.
It is predicted that $s_n \sim A\mu_1^n n^{\gamma-1}$ for a universal critical
exponent $\gamma=\gamma(d)$,  with $\gamma(2) = \frac{43}{32}$ \cite{Nien82,LSW04}
and
$\gamma(3) = 1.156\, 953\, 00(95)$ \cite{Clis17}.
For $d=4$, a result consistent with the predicted behaviour
$s_n \sim A \mu_1^n (\log n)^{1/4}$ have been proved
for the susceptibility of continuous-time weakly self-avoiding walk \cite{BBS-saw4-log}.
For $d \ge 5$, as indicated at \eqref{e:cneta}, it has been proved using the lace expansion that $s_n \sim A\mu_1^n$, so $\gamma=1$
\cite{HS92a}.

For dimensions $d=2,3,4$, the existing upper bounds on $s_n$ remain far from the predicted
behaviour, and have the form
\begin{equation}
\label{e:HWbd}
    s_n \le O(\mu_1^n e^{f_n}).
\end{equation}
In 1962, \eqref{e:HWbd} was proved with $f_n=Bn^{1/2}$ for
any $d\ge 2$ and $B> \pi(2/3)^{1/2}$ if $n$ is sufficiently large \cite{HW62}.
In 1964, this was improved for $d\ge 3$ to $f_n=Qn^{2/(d+2)}\log n$ for some $Q>0$
\cite{Kest64}, \cite[Section~3.3]{MS93}.
More than half a century later, in 2018, the bound was improved
to $f_n=o(n^{1/2})$ \cite{Hutc18} (based on \cite{D-CH13}),
and in the same year it was proved that \eqref{e:HWbd} holds
for $d=2$  with $f_n=n^{0.4979}$ for infinitely
many $n$ (and with $f_n=n^{0.4761}$ for all $n$ on the hexagonal lattice) \cite{DGHM18}.
The distance of these results from the predicted
behaviour $\mu_1^n n^{\gamma-1}$
is an indication of the notorious difficulty in obtaining accurate bounds
for the number of self-avoiding walks.

\subsection{Organisation}

The remainder of the paper is organised as follows.
In Section~\ref{sec:reduction} we reduce the proof of
Theorem~\ref{thm:susceptibility} (which we have seen above implies
our main result Theorem~\ref{thm:dilute}) to Propositions~\ref{prop:Flacebds-new}--\ref{prop:Flbd-new} concerning the infinite-volume
and finite-volume susceptibilities and
Proposition~\ref{prop:AD-new} for the difference between
these two susceptibilities.  It is also in Section~\ref{sec:reduction}
that we prove
Theorem~\ref{thm:chi_torus_good} for the susceptibility and
Theorem~\ref{thm:EL} for the expected length.

A key ingredient for the proofs of Propositions~\ref{prop:Flacebds-new}--\ref{prop:AD-new}
is the estimate on
the near-critical two-point function for $\Z^d$ (Theorem~\ref{thm:mr}) and the
torus plateau that is derived from this near-critical estimate (Theorem~\ref{thm:plateau}).
In Section~\ref{sec:plateauplus} we discuss the plateau theory,
as well as its implications for near-critical estimates for $\Z^d$
quantities, and how this ultimately leads to sharp control of torus convolutions.

The proofs of the Propositions~\ref{prop:Flacebds-new}--\ref{prop:AD-new} are also based on
well-developed theory for the
lace expansion, which is briefly reviewed in Section~\ref{sec:lace}.
In Section~\ref{sec:lacepfs} we obtain estimates on the infinite-volume and
finite-volume lace expansion and prove Propositions~\ref{prop:Flacebds-new}--\ref{prop:Flbd-new}.
Finally, in the most intricate part of our analysis,
in Section~\ref{sec:AD} we use novel estimates for
the lace expansion to directly compare the finite-volume
and infinite-volume susceptibilities and prove Proposition~\ref{prop:AD-new},
together with its partner Proposition~\ref{prop:AD-new-prime} which is used
only in the proof of Theorem~\ref{thm:EL}.

\section{Reduction of proof}
\label{sec:reduction}

In this section, we reduce the proofs of Theorems~\ref{thm:susceptibility} and
\ref{thm:chi_torus_good}
to three propositions, Propositions~\ref{prop:Flacebds-new}--\ref{prop:AD-new}.
The proof of Theorem~\ref{thm:EL} also relies on the extension
of Proposition~\ref{prop:AD-new} given in Proposition~\ref{prop:AD-new-prime}.
Propositions~\ref{prop:Flacebds-new}--\ref{prop:Flbd-new}
are proved in Section~\ref{sec:lacepfs}, and Propositions~\ref{prop:AD-new}
and \ref{prop:AD-new-prime} are proved in Section~\ref{sec:AD}.
Throughout this section, we regard $z$ as a complex variable $z \in \C$.

\subsection{Proof of Theorem~\ref{thm:susceptibility}}

For $z\in \C$ with $|z|< z_c$
we define
\begin{equation}
    F(z) = \frac{1}{\chi(z)}.
\end{equation}
The function $F$ is well understood for $d>4$.  In particular, $F$ and its derivative
with respect to $z$ extend by continuity
to the closed disk $|z|\le z_c$, with $F(z_c)=0$.
The constant $A$ in Theorem~\ref{thm:dilute} is given by
\begin{equation}
\label{e:Adef1}
    A=\frac{1}{-z_cF'(z_c)}.
\end{equation}
The following proposition is part of the well-established theory for $\Z^d$.
A version of
Proposition~\ref{prop:Flacebds-new} for strictly self-avoiding walk
in dimensions $d>4$
is given in \cite{HS92a} (see also \cite{MS93}).
We give an alternate and simpler proof here which uses the near-critical decay of
the two-point function proved in \cite{Slad20_wsaw} to avoid the need for fractional
derivatives as in \cite{HS92a,MS93}.

\begin{prop}
\label{prop:Flacebds-new}
Let $d>4$ and let $\beta$ be sufficiently small.
The derivative $F'(z)$ obeys $F'(z)=-2d+O(\beta)$ for all $|z|\le z_c$, and $-z_cF'(z_c)=1+O(\beta)$.  Also, $F$ obeys the lower bound
\begin{align}
\label{e:Fder}
    |F(z)| & \succ   |1-z/z_c|  \qquad (|z| \le z_c)
    .
\end{align}
\end{prop}

The next proposition is a torus version of Proposition~\ref{prop:Flacebds-new}.
Its proof relies on the plateau for the torus two-point function proved
in \cite{Slad20_wsaw} and discussed in
Section~\ref{sec:plateau}.
For $z\in \C$ we define the meromorphic function
\begin{equation}
    \varphi(z) = \frac{1}{\chi^\T(z)}
    .
\end{equation}
The function $\varphi$ obeys a version of Proposition~\ref{prop:Flacebds-new} in
a smaller disk.

\begin{prop}
\label{prop:Flbd-new}
Let $d>4$ and let $\beta$ be sufficiently small.  Let $\zeta=z_c(1-V^{-1/2})$ and
$U=\{z\in \C : |z|<\zeta\}$.
Then $\varphi'(z) = -2d +O(\beta)$ for all $|z|\le\zeta$,
and $\varphi$ obeys the lower bound
\begin{equation}
\label{e:varphilb}
    |\varphi(z)| \succ
    |1-z/\zeta|
    \qquad (|z| \le \zeta)
    .
\end{equation}
\end{prop}

For $z\in \C$ with $|z|\le z_c$,
 we define
\begin{equation}
\label{e:Deltadef-new}
    \Delta(z)  = \varphi(z) -F(z)
    .
\end{equation}
The following proposition, which is central to the proof of our main result,
provides a quantitative comparison between
the susceptibilities $\chi(z)$ and $\chi^\T(z)$ on $\Z^d$ and on the torus,
as well as on their derivatives.  No such comparison has been performed previously
for any model that has been analysed using the lace expansion,
and the proof of the proposition involves new  methods.
The proof of Proposition~\ref{prop:AD-new}, which also relies heavily
on the torus plateau, is given in Section~\ref{sec:AD}.

\begin{prop}
\label{prop:AD-new}
Let $d>4$ and let $\beta$ be sufficiently small.
For $z\in\C$ with $|z|\le \zeta$,
\begin{align}
    |\Delta(z)|
    & \prec \beta  \frac{r^2+\chi(|z|)}{V}
    .
\end{align}
\end{prop}

Given these three propositions, the proof of Theorem~\ref{thm:susceptibility} is
immediate, as follows.

\begin{proof}[Proof of Theorem~\ref{thm:susceptibility}]
Let $z \in U$.  We wish to prove that
\begin{equation}
    |H(z)| \prec
    \frac{\beta}{V}\frac{1}{|1-z/\zeta|^2(1-|z|/\zeta)}
    +
    \frac{\beta r^2}{V}\frac{1}{|1-z/\zeta|^2}.
\end{equation}
By definition,
\begin{equation}
    H(z) = \frac{1}{\varphi(z)} - \frac{1}{F(z)} = -\frac{\Delta(z)}{F(z)\varphi(z)}.
\end{equation}
By Proposition~\ref{prop:Flacebds-new}, $\chi(|z|) \prec (1-|z|/z_c)^{-1}
\le (1-|z|/\zeta)^{-1}$, and by
Propositions~\ref{prop:Flbd-new} and \ref{prop:AD-new},
\begin{equation}
    |\Delta(z)| \prec \frac{\beta}{V}\frac{1}{1-|z|/\zeta}
    +\frac{\beta r^2}{V},
    \qquad
    |\varphi(z)| \succ |1-z/\zeta|,
\end{equation}
so it remains only to prove that
$|F(z)| \succ |1-z/\zeta|$.  But $|F(z)| \succ |1-z/z_c|$ by
Proposition~\ref{prop:Flacebds-new}.
Thus, since $\zeta/z_c \ge \frac 12$ for $V \ge 4$, it suffices to observe that
\begin{equation}
\label{e:disk}
    |z_c-z| \ge |\zeta-z| \qquad (z\in U)
\end{equation}
To prove \eqref{e:disk}, we write $z=x+iy$.  Both $z_c-z$ and $\zeta-z$ have imaginary part $-iy$.
The real parts obey $z_c -x \ge \zeta-x$, and this gives the claimed inequality
and completes the proof.
\end{proof}

\subsection{Susceptibility and expected length}
\label{sec:ELpf}

Theorem~\ref{thm:chi_torus_good} for the susceptibility
is an immediate consequence of Proposition~\ref{prop:AD-new}, as follows.

\begin{proof}[Proof of Theorem~\ref{thm:chi_torus_good}]
Let $|z| \le \zeta$.
It follows from the definition of $\Delta$ in \eqref{e:Deltadef-new} and from Proposition~\ref{prop:AD-new}
that
\begin{align}
    \chi(z) & = \chi^\T(z) \left(1 + \frac{\Delta(z)}{F(z)} \right) = \chi^\T(z) (1+\chi(z)\Delta(z))
    \nnb & =
    \chi^\T(z) \left( 1 + O\Big(\beta \frac{\chi(z)(r^2 +\chi(|z|))}{V} \Big) \right),
\end{align}
which proves Theorem~\ref{thm:chi_torus_good}.
\end{proof}

Next we consider the expected length and prove Theorem~\ref{thm:EL}.
We begin with an elementary lemma concerning the susceptibility on $\Z^d$.
This is a kind of Abelian theorem in which the asymptotic behaviour of the
coefficients of a generating function are used to determine the behaviour of the
generating function.

\begin{lemma}
\label{lem:ELZd}
Let $d>4$ and let $\beta$ be sufficiently small.  Then for $z \in [0,z_c)$,
\begin{align}
    \chi (z)  & =
    \frac{A}{1-z/z_c}
    + O\bigg(\frac{\beta}{(1-z/z_c)^{\frac 12 (6-d)_+}}\bigg)
    \\
    z\partial_z \chi (z)  & =
    \frac{Az/z_c}{(1-z/z_c)^2}
     + O\bigg(\frac{\beta}{(1-z/z_c)^{\frac 12 (8-d)_+}}\bigg) ,
\end{align}
where in the error terms $x_+ = \max\{0,x\}$, although when $x=0$
($d=6$ for $\chi$ and $d=8$ for the derivative)
the error term contains $|\log(1-z/z_c)|$ instead.
\end{lemma}

\begin{proof}
Let $z \in [0,z_c)$ and recall that $\mu=z_c^{-1}$.
We only prove the second statement as the first is similar.
By \eqref{e:cnasy},
\begin{align}
    z\partial_z \chi (z)  & = \sum_{n=1}^\infty n c_n z^n =
    A\sum_{n=1}^\infty (\mu z)^n(n+O(\beta n^{1-(d-4)/2}))
    \nnb & =
    \frac{Az/z_c}{(1-z/z_c)^2} + \beta \sum_{n=1}^\infty (\mu z)^n O(n^{(6-d)/2}).
\label{e:ELpf}
\end{align}
It is an elementary fact that $\sum_{n=1}^\infty t^n n^a = O((1-t)^{-(1+a)})$ for
$t \in [0,1)$ and $a > -1$, with a logarithmic divergence for $a=-1$ and a uniform bound
for $a<-1$.  From this we see that the last term in \eqref{e:ELpf}
is bounded as claimed.
\end{proof}

For the expected length we will also apply the following proposition.
Its proof is given in Section~\ref{sec:AD}.

\begin{prop}
\label{prop:AD-new-prime}
Let $d>4$ and let $\beta$ be sufficiently small.
For $z\in\C$ with $|z|\le \zeta$,
\begin{align}
\label{e:Delp}
    |z\Delta'(z)|
     \prec \beta  \frac{\chi(|z|)(r^2+\chi(|z|))}{V}
    .
\end{align}
\end{prop}

\begin{proof}[Proof of Theorem~\ref{thm:EL}]
Let $z \in [\frac 12 z_c,z_c(1-V^{-1/2})]$.
By definition,
\begin{equation}
\label{e:expect_L}
	\E_{z}^{\T}L = \frac{z\partial_z \chi^{\T}  (z)}{\chi^\T(z)}
    = \frac{- z \varphi'(z)}{\varphi(z)}.
\end{equation}
Since $\varphi=F+\Delta$ and $F=1/\chi$, this becomes
\begin{align}
	\E_{z}^{\T}L
    &=
    \frac{-z F'(z)[1+\Delta'(z)/F'(z)]}{F(z)[1+\Delta(z)/F(z)]}
    \nnb
    &=
    \frac{z\partial_z\chi(z)}{\chi(z)}
    \frac{1+\Delta'(z)/F'(z)}{1+\Delta(z)/F(z)}.
\lbeq{ELpf2}
\end{align}
For the second ratio on the right-hand side, we
recall from Proposition~\ref{prop:Flacebds-new} that $F'(z) = -2d +O(\beta)$.
Also, Propositions~\ref{prop:AD-new} and \ref{prop:AD-new-prime} give bounds
on $\Delta(z)$ and $\Delta'(z)$  (the restriction that $z \ge \frac 12 z_c$ renders
the factor $z$ on the left-hand side of \eqref{e:Delp} irrelevant).
Since $F=1/\chi$, these facts imply that
\begin{align}
	\frac{1+\Delta'(z)/F'(z)}{1+\Delta(z)/F(z)}
    &=
    1 + O\Big(\beta\frac{\chi(z)(r^2+\chi(z))}{V} \Big)
    \nnb & =
    1+O\Big( \beta\frac{1}{V (1-z/z_c)^2} \Big)
    +O\Big( \beta\frac{r^2}{V (1-z/z_c)} \Big)
    .
\end{align}
The first ratio on the right-hand side of \refeq{ELpf2} is the expected length for $\Z^d$,
which can be analysed using Lemma~\ref{lem:ELZd}.  To lighten the notation,
we temporarily write $\varepsilon = 1-z/z_c$.  Then Lemma~\ref{lem:ELZd} leads to
\begin{equation}
    \frac{z\partial_z\chi(z)}{\chi(z)}
    =
    \frac{z}{z_c} \frac{1}{\varepsilon}
    \left( 1 +
    \beta O(\varepsilon^{1-\frac12(6-d)_+} + \varepsilon^{2-\frac12(8-d)_+})
    \right).
\end{equation}
The term with power $1-\frac12(6-d)_+$ dominates the other one.
The desired result then follows by assembling the above equations.
\end{proof}

\section{The torus plateau and its consequences}
\label{sec:plateauplus}

This section contains
essential ingredients for our proofs of
Propositions~\ref{prop:Flacebds-new}--\ref{prop:AD-new}, especially Theorems~\ref{thm:mr}--\ref{thm:plateau} which provide the means to obtain the other
ingredients.
These two theorems, which give an estimate for the near-critical two-point function on $\Z^d$
and establish the existence of a ``plateau'' for the torus two-point function, are proved in
\cite{Slad20_wsaw}.
Related results for percolation are obtained in \cite{HMS22}, where the role of the
plateau in the analysis of high-dimensional torus percolation is emphasised.

\subsection{The torus plateau}
\label{sec:plateau}

The $\Z^d$ and torus \emph{two-point functions} are defined by
\begin{equation}
\label{e:Gzdef}
    G_{z}(x) = \sum_{n=0}^\infty c_{n}(x) z^n ,
    \qquad
    G_{z}^\T(x) = \sum_{n=0}^\infty c_{n}^\T(x) z^n .
\end{equation}
The susceptibilities are then by definition equal to
\begin{equation}
    \chi(z)=\sum_{x\in \Z^d} G_{z}(x)
    ,
    \qquad
    \chi^\T(z)=\sum_{x\in \Z^d} G_{z}^\T(x)
    .
\end{equation}
The following theorem concerning the near-critical decay of the
two-point function is
\cite[Theorem~1.1]{Slad20_wsaw}.
To simplify the notation, for $x\in \Z^d$ we use the notation
\begin{equation}
    \veee{x} = \max\{1,\|x\|_\infty\}
    .
\end{equation}

\begin{theorem}
\label{thm:mr}
Let $d>4$ and let $\beta$ be sufficiently small.
There are constants $c_0>0$ and $c_1\in (0,1)$
such that for all $z\in (0, z_c)$ and $x\in\Z^d$,
\begin{equation}
\label{e:Gmr}
    G_{z}(x) \le c_0 \frac{1}{\veee{x}^{d-2}}
    e^{-c_1 m(z)\|x\|_\infty}.
\end{equation}
The mass has the asymptotic form
\begin{equation}
\label{e:massasy}
    m(z) \sim c(1-z/z_c)^{1/2} \qquad (z \to z_c),
\end{equation}
with constant $c=(2d)^{1/2} + O(\beta)$.
\end{theorem}

A matching lower bound $\xvee^{-(d-2)}$ is also proved for $z=z_c$; in fact
an asymptotic formula proportional to $\|x\|_{2}^{-(d-2)}$ has been proved \cite{Slad20_lace}
(see also \cite{BHH19,HHS03,Hara08} for closely related results).
In \cite{Slad20_wsaw},
Theorem~\ref{thm:mr} is applied to prove that the torus
two-point function has a ``plateau'' in the sense of the following theorem,
which is \cite[Theorem~1.2]{Slad20_wsaw}.
For notational
convenience, in \eqref{e:plateau} (and also elsewhere) we evaluate a $\Z^d$ two-point function at a point $x\in\T_r^d$
with the understanding that in this case we identify
$x$ with a point in $[-r/2,r/2)^d \cap \Z^d$.

\begin{theorem}
\label{thm:plateau}
Let $d>4$ and let $\beta$ be sufficiently small.
There are constants $c_i>0$ such that for all $x \in \T_r^d$,
\begin{equation}
\label{e:plateau}
    G_{z}(x) +c_1 \frac{\chi(z)}{V}
    \le
    G^{\T}_{z}(x)
    \le
    G_{z}(x) +c_2 \frac{\chi(z)}{V},
\end{equation}
where the upper bound holds for all $r \ge 3$ and all $z \in (0,z_c)$, whereas
the lower bound holds provided that
$z \in [ z_c-c_3r^{-2},   z_c-c_4\beta^{1/2}r^{-d/2}]$.
\end{theorem}

By \eqref{e:chiZdasy}, $\chi(z_c(1-V^{-p})) \sim A V^p$.
Informally, the plateau states that
for $z_c-z$ of order $V^{-p}$ with $p \in [\frac{2}{d},\frac 12]$,
the torus two-point function evaluated at $x$ has
the $|x|^{-(d-2)}$ decay of the critical $\Z^d$ two-point function until its value reaches
order $V^{-(1-p)}$, beyond which it fixates at
this ``plateau'' value. This quantifies the distance at which the two-point function begins to ``feel'' it is on
the torus.
From the plateau, as noted in \cite[Corollary~1.3]{Slad20_wsaw}, it is easy to obtain
the lower bound of \eqref{e:chiupperlower} simply by neglecting the
$G_z(x)$ term on the left-hand side of \refeq{plateau} and then summing over $x \in \T_r^d$
to obtain
\begin{equation}
	c_1\chi(z)
    \leq
    \chi^\T(z)
    \le \chi(z).
\end{equation}
The upper bound holds without restriction since $c_n^\T \le c_n$, and the lower bound
holds for $d>4$, small $\beta>0$, and for
$z \in  [ z_c- c_3 r^{-2}, z_c - c_4\beta^{1/2} r^{-d/2}]$.

\subsection{A convolution lemma}

We write $*$ for the convolution of functions $f,g$ defined on $\Z^d$:
\begin{equation}
\label{e:Zconv}
    (f*g)(x) = \sum_{y \in \Z^d}f(x-y)g(y).
\end{equation}
We use the following lemma repeatedly to bound convolutions.
It extends \cite[Proposition~1.7]{HHS03} by including the possibility of exponential
decay as well as including the case $a+b \leq d$.  The important $\nu$ values are
small or zero.

\begin{lemma}
\label{lem:conv-nu}
Let $\nu \ge 0$.
Suppose that $f, g : \Z^d\to \C$ satisfy
$| f(x) | \leq \veee{x}^{-a} e^{-\nu\|x\|_\infty}$ and
$| g(x) | \leq \veee{x}^{-b} e^{-\nu\|x\|_\infty}$
with $a \geq b \ge 0$.  There is a constant $C$ depending on $a, b, d$
such that
    \begin{equation}
    \label{e:fg.2}
        \bigl | (f*g)(x) \bigr | \leq
   \begin{cases}
    C  \veee{x}^{-b}  e^{-\nu\|x\|_\infty}
   & (a > d )
   \\
   C \veee{x}^{d-(a + b)} e^{-\nu\|x\|_\infty}  &
       (a < d \mbox{ and } a+b>d)
   \\
   C \nu^{a+b-d}   e^{- \frac 12 \nu \|x\|_\infty}
   &  (a+b < d)
     \\
   C|\log (\nu\veee{x})|
    \1_{\nu \xvee \leq 1} + C
   e^{-  \nu \|x\|_\infty}
   \1_{\nu \xvee\geq 1}
   &(a< d \mbox{ and }a+b =d).
   \\
    C  \log(1+\xvee)\xvee^{-b}e^{-\nu\|x\|_\infty}
   &(a= d \mbox{ and }a+b >d).
   \\
   C|\log (\nu\veee{x})|\1_{\nu \xvee \leq 1} + C\log(1+\xvee)
   e^{- \nu \|x\|_\infty}\1_{\nu \xvee\geq 1}
   &(a= d \mbox{ and } b=0
   ).
   \end{cases}
    \end{equation}
Note that the fourth and sixth bounds, for which $a+b=d$, are infinite when $\nu=0$.
\end{lemma}

\begin{proof}
In the proof, we write simply $|x|$ in place of $\|x\|_\infty$.
By definition,
\begin{equation}
    | (f*g)(x) | \leq
    \sum_{y: |x-y| \leq |y|}
    \frac{e^{-\nu|x-y|}}{\xyvee^{a}} \frac{e^{-\nu|y|}}{\yvee^{b}}
    +
    \sum_{y: |x-y| > |y|}
    \frac{e^{-\nu|x-y|}}{\xyvee^{a}} \frac{e^{-\nu|y|}}{\yvee^{b}}.
\end{equation}
Using $a \geq b$ and the change of variables $z=x-y$ in the second
term, we see that
\begin{equation}
\label{e:fg.3}
    | (f*g)(x) |  \leq
    2\sum_{y: |x-y| \leq |y|}
    \frac{e^{-\nu|x-y|}}{\xyvee^{a}} \frac{e^{-\nu|y|}}{\yvee^{b}}
    .
\end{equation}

\medskip \noindent \emph{Case of} $a>d$.
The restriction $|x-y| \leq |y|$ in \eqref{e:fg.3}
ensures that $|y| \geq \frac{1}{2}|x|$, so
 by using the triangle inequality for the exponentials we find
that
\begin{equation}
    | (f*g)(x) |  \leq
    \frac{2^{b+1} e^{-\nu |x|}}{\xvee^{b}}
    \sum_{y: |x-y| \leq |y|}
    \frac{1}{\xyvee^{a}}
    .
\end{equation}
The sum converges and we obtain the desired estimate.

\medskip \noindent \emph{Case of} $a < d$ and $a+b>d$.
We divide the sum in \eqref{e:fg.3}
according to whether
$|y| \leq \frac{3}{2}|x|$ or
$|y| \geq \frac{3}{2}|x|$.  The contribution due to the
first range of $y$ is bounded above by
\begin{equation}
\label{e:fg.7}
	\frac{2^{b+1} e^{- \nu |x|}}{\xvee^b}
    \sum_{y: |x-y| \leq 3|x|/2}
    	\frac{1}{\xyvee^{a}}
    	\leq
    	\frac{C}{\xvee^b} \xvee^{d-a} e^{- \nu |x|}.
\end{equation}
When $|y| \geq \frac{3}{2}|x|$, we have $|y-x| \geq |y|-|x| \geq \frac 13 |y|$,
so the sum over the second range
of $y$ is bounded above by
\begin{equation}
\label{e:fg.8}
	3^a \cdot 2 \sum_{y: |y| \geq 3|x|/2}
    	 \frac{e^{-\nu|y|}}{\yvee^{a+b}}
    	.
\end{equation}
We extract the exponential factor using the lower limit of summation,
which leaves the tail of a convergent series.  This gives the desired result.

\medskip \noindent \emph{Case of} $a+b< d$.
In particular, this implies that $a <d$. By \eqref{e:fg.3},
since $|y| \geq \frac12 |x|$, we have
\begin{align}
	| (f*g)(x) |  &\leq 2e^{-\frac\nu2 |x|}\sum_{y: |x-y| \leq |y|}
    \frac{e^{-\nu|x-y|}}{\veee{x-y}^{a+b}}
    \leq 2e^{-\frac\nu2 |x|} \sum_{y \in \Z^d}\frac{e^{-\nu |y|}}{\yvee^{a+b}}.
\end{align}
This last sum is dominated by a multiple of
\begin{align}
	\int_{1}^\infty e^{-\nu t}t^{d-1-(a+b)} \D t
    = \nu^{a+b-d}\int_{\nu}^\infty e^{-t}t^{d-1-(a+b)} \D t,
\end{align}
which is bounded above by a multiple of $\nu^{a+b-d}$.

\medskip \noindent \emph{Case of} $a<d$ and $a+b = d$.
We can adapt the proof of the case with $a<d$ and $a+b > d$ by
again dividing the sum according to whether $|y| \leq \frac32 |x|$ or $|y|\geq \frac 32|x|$.
Since $a<d$, \eqref{e:fg.7} remains unchanged and gives a contribution of the form
\begin{equation}
     \veee{x}^{d-a-b}e^{-\nu |x|} =  e^{-\nu |x|}.
\end{equation}
For $a+b = d$, in \eqref{e:fg.8} we can take $\yvee \ge \xvee$ instead of
$|y| \ge 3|x|/2$
in the summation restriction and the sum is then dominated by
\begin{align}
\label{e:fg.11}
	\int_{\xvee}^\infty \frac{e^{-\nu t}}{t} \D t
    &= \int_{\nu \xvee}^\infty \frac{e^{-t}}{t} \D t
    \leq |\log(\nu\xvee)|\1_{\nu \xvee \leq 1}
    + \frac{e^{-\nu |x|}}{\nu \xvee} \1_{\nu\xvee\geq 1}.
\end{align}

\medskip \noindent \emph{Case of} $a=d$ and $a+b \geq d$.
We proceed as in the previous case, again considering separately the
cases $|y| \leq \frac32 |x|$ or $|y|\geq \frac 32|x|$.
With $a=d$, \eqref{e:fg.7} becomes instead
\begin{equation}
\label{e:fg.12}
	\frac{2^{b+1}e^{-\nu |x|}}{\xvee^b}
    \sum_{y: |x-y| \leq 3|x|/2}
    	\frac{1}{\xyvee^{d}} \leq C\frac{\log(1 + \veee{x})}{\xvee^b}e^{-\nu |x|}.
\end{equation}
Exactly as in \eqref{e:fg.11},
when $a+b = d$, i.e., when $b=0$, the second range of $y$ gives a contribution
\begin{equation}
|\log (\nu\veee{x})|
\1_{\nu \xvee \leq 1} + \frac{e^{-\nu |x|}}{\nu \xvee}\1_{\nu \xvee\geq 1},
\end{equation}
and with \eqref{e:fg.12} this proves the case $a=d$, $b=0$ of \eqref{e:fg.2}.
If instead $a=d$ and $b>0$
then \eqref{e:fg.8} is bounded by
\begin{equation}
 \frac{1}{\veee{x}^{b}}
 e^{-\nu |x|} \prec \frac{\log(1 + \veee{x})}{\xvee^b}e^{-\nu |x|}.
\end{equation}
This gives the desired result and completes the proof.
\end{proof}

\subsection{Near-critical estimates}

For later use, we gather here some  consequences of Theorem~\ref{thm:mr}.
A minor observation is that
if $x \in \T_r^d$ is regarded as a point in $[-\frac r2, \frac r2)^d \cap \Z^d$
then
$\frac 12 \|ru\|_\infty\le \|x+ru\|_\infty \le \frac 32 \|ru\|_\infty$
uniformly in $x \in \T_r^d$
and in nonzero $u\in \Z^d$, since
\begin{equation}
\label{e:xulb}
    \|x+ru\|_\infty \ge \|ru\|_\infty - \frac r2 \ge  \|ru\|_\infty - \frac 12\|ru\|_\infty
    = \frac 12 \|ru\|_\infty
    ,
\end{equation}
\begin{equation}
\label{e:xuub}
    \|x+ru\|_\infty \le  \frac r2 + \|ru\|_\infty  \le  \frac 12\|ru\|_\infty + \|ru\|_\infty
    = \frac 32 \|ru\|_\infty.
\end{equation}
We write $a \in \R$ as $a=a_+-a_-$ with $a_+ = \max\{a,0\}$ and $a_- = -\min\{a,0\}$.

\begin{lemma}
\label{lem:unifmassint}
Let $r \ge 1$, $a \in \R$ and $\nu >0$.  There
is a constant $C_a$ (independent of $\nu,r$)
such that, for $x\in \T_r^d$,
\begin{align}
	\sum\limits_{u \in\Z^d : u \neq 0} \frac{1}{\veee{x + r u}^{d-a}}e^{- \nu \|x+ru\|_\infty}
	&\leq
    C_a
    e^{-\frac 14 \nu r}
\times
    \begin{cases}
     r^{-(d+a_-)} \nu^{-a_+}  & (a\neq 0)
     \\
     r^{-d}|\log (\nu r)|  & (a= 0).
     \end{cases}
\end{align}
\end{lemma}

\begin{proof}
For $a<0$, we simply note that the sum is convergent without the exponential factor, so using \eqref{e:xulb} and extracting the factoring $r^{d+a_-}$ from the sum gives the result.

Suppose that $a\ge 0$. It follows
from
\eqref{e:xulb} that for any nonzero $u \in \Z^d$,
$\|x + r u\|_\infty \geq \frac12 \|ru\|_\infty$ and thus
\begin{align}
	\sum\limits_{u \neq 0} \frac{1}{\veee{x + r u}^{d-a}}e^{-\nu \|x+ru\|_\infty}
	&\leq
    \sum\limits_{u \neq 0} \frac{1}{  (\frac 12 \|ru\|_\infty)^{d-a}}
    e^{-\frac{1}{2} \nu \|ru\|_\infty}\nnb
	&\le 2^{a-d}e^{-\frac 14 \nu r}\sum\limits_{N = 1}^\infty \sum\limits_{u:\|u\|_\infty =N}
    \frac{1}{\|ru\|_\infty^{d-a}}e^{-\frac 14 \nu \|ru\|_\infty}
    \nnb
    &\prec r^{a-d}e^{-\frac 14 \nu r}
    \sum\limits_{N = 1}^\infty N^{d-1 - d + a }e^{-\frac 14 \nu rN}
	.
\end{align}
We bound the sum on the right-hand side by an integral to obtain an upper bound
which is a constant multiple of
\begin{align}
	r^{a-d}e^{-\frac 14 \nu r}
    \int_1^\infty u^{a-1} e^{-\frac 14 \nu ru} \D u
    & =
    \frac{1}{\nu^a r^d} e^{-\frac 14 \nu r} \int_{\nu r}^\infty t^{a-1}e^{-t/4} \D t
    .
\end{align}
The integral is uniformly bounded if $a>0$ and behaves as $|\log (\nu r)|$
for $a=0$.
This concludes the proof.
\end{proof}

For $x \in \Z^d$ we define the open \emph{bubble diagram}
and the open \emph{triangle diagram}  by
\begin{equation}
    \bubble_z(x) = (G_z*G_z)(x),
    \qquad
    {\sf T}_z(x) = (G_z*G_z*G_z)(x).
\end{equation}
As the following lemma shows, the critical bubble diagram is bounded for all $d>4$.
However the critical triangle
diagram\footnote{It may appear strange to see the triangle diagram appearing
in a model with upper critical dimension $4$.  We use it in the proof of
Proposition~\ref{prop:AD-new} in Section~\ref{sec:AD}.}
is finite only in dimensions $d>6$, and the lemma gives a bound on the rate of divergence
when $d \le 6$.

\begin{lemma}
\label{lem:G3}
Let $d >4$ and let $\beta$ be sufficiently small.  For $x \in \Z^d$ and $z \in (0,z_c)$,
and with $c_1$ the constant from \eqref{e:Gmr},
\begin{align}
\label{e:bubblebd}
    \bubble_z(x) & \prec
    \frac{e^{-c_1 m(z) |x|}}{\veee{x}^{d-4}},
\\
\label{e:GGG}
    {\sf T}_z(x)
    &\prec
    \begin{cases}
    m(z)^{-(6-d)} e^{-\frac {1}{2} c_1 m(z)\|x\|_\infty} & (d< 6)
    \\
    |\log (c_1 m(z)\veee{x})|
    \1_{c_1 m(z) \xvee  \leq 1} +
    e^{-  c_1 m(z) \|x\|_\infty}
   \1_{c_1 m(z) \xvee\geq 1}
     & (d= 6)
    \\
    \veee{x}^{-(d-6)}e^{- c_1 m(z) \|x\|_\infty} & (d> 6).
    \end{cases}
\end{align}
\end{lemma}

\begin{proof}
We write $m=m(z)$.  The bound on the bubble diagram is an immediate consequence of
the bound on $G_z(x)$ from \eqref{e:Gmr} together with the convolution estimate
\eqref{e:fg.2} with $a=b=d-2$.
For the triangle diagram, we use \eqref{e:Gmr} and \eqref{e:bubblebd} to obtain
\begin{align}
    {\sf T}_z(x) = (G_{z}* \bubble_{z} )(x)
    & \prec
    \sum_{v\in \Z^d}
    \frac{e^{- c_1 m  \|v-x\|_\infty}}{\veee{v-x}^{d-2}}
    \frac{e^{-c_1m \|v\|_\infty}}{\veee{v}^{d-4}}
    .
\end{align}
Now we apply Lemma~\ref{lem:conv-nu} with $a=d-2$ and $b=d-4$, so $a+b=2d-6$.
\end{proof}

Bounds expressed in terms of the mass $m(z)$, such as \eqref{e:GGG},
 can also be expressed in terms of the
susceptibility $\chi(z)$ since
\begin{equation}
\label{e:mchi}
    \frac{1}{m(z)^2} \prec \chi(z).
\end{equation}
To prove \eqref{e:mchi}, we first
fix any $z_1 \in (0,z_c)$.  For $z \le z_1$, since $m$ is decreasing and
since $1=\chi(0) \le \chi(z)$, we have $m(z)^{-2} \le m(z_1)^{-2} \le
m(z_1)^{-2}\chi(z)$ and the desired upper bound
follows for $z \in (0,z_1]$.
We can choose $z_1$ close enough to $z_c$ that $m(z)^{-2}$ and $\chi(z)$ are
comparable for $z\in (z_1,z_c)$, since
each is asymptotic to a multiple of $(1-z/z_c)^{-1}$
by \eqref{e:massasy} and \eqref{e:chiZdasy}.
In particular
for $z_1$ close enough to $z_c$ there exists $C$ such that $m(z)^{-2} \leq C \chi(z)$ for  $z\in[z_1,z_c)$.

We define the function $\Gamma_z : \Z^d \to \R$  by
\begin{align}
\label{e:def-gamma}	
	\Gamma_z(x) &= \sum_{u \in \Z^d} G_z(x+ru)  \qquad (x \in \Z^d).
\end{align}
Note that $\Gamma_z(x) = \Gamma_z(y)$ whenever $x,y\in \Z^d$ project
to the same torus point.
The function $\Gamma_z$ arises naturally since
\begin{equation}
\label{e:GGam}
    G^\T_z(x) \le \Gamma_z (x) \qquad (x\in \T_r^d),
\end{equation}
which follows from the fact
that the lift of a torus walk to $x$ must end at a point $x+ru$ in $\Z^d$,
and the lift of the walk can have no more intersections than the walk itself.
Also,
\begin{equation}
    \chi(z) = \sum_{x \in \T_r^d} \Gamma_z(x).
\end{equation}

We write $\star$ for the convolution of functions $f,g$ on the torus:
\begin{equation}
\label{e:Tconv}
    (f\star g)(x) = \sum_{y\in \T_r^d}f(x- y)g(y),
\end{equation}
where on the right-hand side the subtraction is
on the torus, i.e., modulo $r$.
Note that we make a distinction between
$*$ for convolution in $\Z^d$ and $\star$ for convolution in $\T_r^d$.
The combination of \eqref{e:GGam} with the estimate \eqref{e:plateau-ub}
from the next lemma provides the
proof of the plateau upper bound in \eqref{e:plateau} with $c_2=C$.
Also, again using \eqref{e:GGam}, the  lemma gives bounds on the torus convolutions
$(G^\T_z \star G^\T_z )(x)$ and $(G^\T_z \star G^\T_z \star G^\T_z )(x)$.

\begin{lemma}
\label{lem:GT3}
Let $d >4$ and let $\beta$ be sufficiently small.
For $x \in \T^d_r$ and $z \in [0,z_c]$,
\begin{align}
\label{e:plateau-ub}
	\Gamma_z(x) &\leq G_z(x) + C\frac{\chi(z)}{V} , \\
\label{e:GGstar}
    (\Gamma_z \star \Gamma_z )(x)
    &\le \bubble_z(x) + C\frac{\chi(z)^2}{V}
    ,
\\
\label{e:GGGstar}
    (\Gamma_z \star \Gamma_z \star \Gamma_z)(x)
    &\le
    {\sf T}_z(x) + C\frac{\chi(z)^3}{V}
    .
\end{align}
\end{lemma}

\begin{proof}
We write $m=m(z)$.
For \eqref{e:plateau-ub}, we separate the $u = 0$ term from the sum in \eqref{e:def-gamma}
and then apply \eqref{e:Gmr}, Lemma~\ref{lem:unifmassint}
with $a=2$, and \eqref{e:mchi}, to obtain
\begin{align}
\label{e:GammachiV}
	\Gamma(x) &\leq  G_z(x) + \sum_{u \neq 0}
	\frac{c_0}{\veee{x + r u}^{d-2}}e^{-c m |x+ru|}
	\leq G_z(x) + C\frac{\chi(z)}{V}.
\end{align}

For \eqref{e:GGstar}, we first observe that
\begin{align}
    (\Gamma_z \star \Gamma_z )(x)
    & =
    \sum_{y \in \T_r^d} \sum_{u,v \in \Z^d} G_z(y+ru) G_z(x-y + rv)
    \nnb & =
    \sum_{w\in\Z^d} \sum_{y \in \T_r^d} \sum_{v \in \Z^d} G_z(y-rv+rw) G_z(x-y + rv)
    \nnb & =
    \sum_{w\in \Z^d} (G_z*G_z)(x-rw)
    =
    \sum_{w\in \Z^d} \bubble_z(x+rw),
\label{e:GamGam}
\end{align}
where in the second equality we replace $u$  by $w-v$, and in the third we observe
that $y-rv$ ranges over $\Z^d$ under the indicated summations.
We again separate the $w=0$ term, which is $\bubble_z(x)$, and now use
\eqref{e:bubblebd}, Lemma~\ref{lem:unifmassint} and \eqref{e:mchi} to see that
\begin{align}
\label{e:Bbd}
    \sum_{w\neq 0} \bubble_z(x+rw)
     &\prec
    \sum_{w\neq 0}  \frac{e^{- c_1m  \|x+rw\|_\infty}}{\veee{x+rw}^{d-4}}
    \prec
    \frac{\chi(z)^2}{V}.
\end{align}

For \eqref{e:GGGstar}, as in \eqref{e:GamGam} we find that
\begin{align}
(\Gamma_z \star \Gamma_z \star \Gamma_z)(x)
&
=
    \sum_{w\in \Z^d} {\sf T}_z(x+rw)
.
\label{e:EMconv}
\end{align}
We extract the $w=0$ term, which is ${\sf T}_z(x)$.
For $d>6$, which has $\veee{x}^{6-d}e^{- c_1m |x|}$
in \eqref{e:GGG}, it follows from
Lemma~\ref{lem:unifmassint} and \eqref{e:mchi} that
\begin{equation}
    \sum_{w\neq 0}  \frac{e^{- c_1m  \|x+rw\|_\infty}}{\veee{x+rw}^{d-6}}
    \prec
    \frac{\chi(z)^3}{V}.
\end{equation}
For $4 < d < 6$, we use \eqref{e:GGG} and
Lemma~\ref{lem:unifmassint} (with $a=d$)
 to obtain an upper bound
\begin{equation}
    \frac{1}{m^{6-d}}\sum_{w\neq 0}e^{-\frac {1}{2} c_1m \|x+rw\|_\infty}
    \prec
    \frac{1}{m^{6-d}} \frac{1}{V} \frac{1}{m^d}
    \prec
    \frac{\chi(z)^3}{V}.
\end{equation}
For the final case $d=6$, we use \eqref{e:GGG} with $\nu  = c_1 m(z)$ to see that
\begin{align}
	\sum_{w \neq 0} {\sf T}_z(x+rw)
	&\prec \sum_{w \neq 0}
    |\log(\nu\veee{x+rw})|
    \1_{\nu \veee{x+rw} \leq 1}
    + \sum_{w \neq 0}e^{-\nu |x+rw|}
    \1_{\nu \veee{x+rw} >1} .
\end{align}
By Lemma~\ref{lem:unifmassint} with $d=a=6$,
the second sum is dominated by $V^{-1}\chi(z)^3$.
For the first sum,
we see from \eqref{e:xulb} that
\begin{equation}
    \nu \veee{x+rw} \ge \nu  \|x+rw\|_\infty \ge \frac 12  \nu r \|w\|_\infty
\end{equation}
and hence the first sum is dominated
by a multiple of
\begin{align}
    \int_{\nu r \|w\|_\infty/2 \le 1}
    |\log ( \nu r \|w\|_\infty/2)| \D w
    &
    =
    2^6 (\nu r)^{-6} \int_{\|v\|_\infty \le 1}
    |\log (\|v\|_\infty)| \D v
     \prec
    \frac{\chi(z)^3}{V},
\end{align}
since the logarithm is integrable.
This completes the proof.
\end{proof}

\section{The lace expansion}
\label{sec:lace}

Since its introduction by Brydges and Spencer in 1985 \cite{BS85},
the lace expansion has been discussed at length and derived many times in the literature.
In this section, we summarise
the definitions, formulas and estimates that we need for the proofs of
Propositions~\ref{prop:Flacebds-new}--\ref{prop:AD-new}.
This is well-established material and is as in the original paper \cite{BS85}.
Our presentation follows
\cite[Sections~3.2--3.3]{Slad06}, where the proofs we omit here are presented in detail,
and we refer to \cite{Slad06} in the following.
Although previous literature has developed the lace expansion in the setting of $\Z^d$,
as we discuss below it applies without
modification also to the torus.

\subsection{Graphs and laces}

\begin{defn}
\label{def-graph}
{\rm (i)}
Given an interval $I = [a,b]$ of positive integers, an \emph{edge} is a pair
$\{ s, t\}$
of elements of $I$, often written simply as $st$ (with $s<t$).
A set of edges (possibly the empty set) is called a \emph{graph}.
\newline
{\rm (ii)}
A graph $\Gamma$ is \emph{connected}\footnote{This
definition of connectivity is not the usual notion
of path-connectivity in graph
theory.  Instead, connected graphs are those
for which $\cup_{st \in \Gamma}(s,t)$ is equal to the connected
interval $(a,b)$.  This is the useful definition of connectivity for the lace expansion.}
if both $a$ and $b$ are
endpoints of edges in $\Gamma$, and if in addition, for any $c \in (a,b)$,
there is an edge $st \in \Gamma$ such that $s < c < t$.
\end{defn}

\begin{defn}
\label{def-lace}
A \emph{lace}
is a minimally connected graph, i.e., a connected graph for which
the removal of any edge would result in a disconnected graph.  The set of
laces on $[a,b]$
is denoted by $\Lcal [a,b]$, and the set of laces on
$[a,b]$ which consist of exactly $N$ edges is denoted $\Lcal ^{(N)} [a,b]$.
\end{defn}

A lace
$L \in \Lcal^{(N)}[a,b]$ can be written by listing its edges as $L = \{ s_1t_1, \ldots , s_Nt_N\}$,
with $s_l < t_l$ for each $l$.  For $N=1$, we simply have
$a=s_1<t_1=b$.   For $N \geq 2$, a graph is a lace
$L \in \Lcal^{(N)}[a,b]$
if and only if the edge endpoints are ordered according to
\begin{equation}
\label{e:LNij}
    a=s_1<s_2, \quad
    s_{l+1}<t_l\leq s_{l+2} \hspace{5mm} (l=1,\ldots,N-2), \quad
    s_N < t_{N-1} <t_N = b
\end{equation}
(for $N=2$ the vacuous middle inequalities play no role); see
Fig.~\ref{fig:graphs}.  Thus $L$ divides $[a,b]$ into
$2N-1$ subintervals:
\begin{equation}
\label{e:LNint}
    [s_1,s_2], \; [s_2,t_1], \; [t_1,s_3], \; [s_3,t_2], \;
    \ldots \; ,[s_N,t_{N-1}], \; [t_{N-1},t_N].
\end{equation}

\begin{defn}
\label{def:lace_presc}
Given a connected graph $\Gamma$ on $[a,b]$,
the following prescription associates to $\Gamma$ a unique
lace ${\sf L}_\Gamma \subset \Gamma$:  The lace ${\sf L}_\Gamma$ consists
of edges $s_1 t_1, s_2 t_2, \ldots$, with $t_1,s_1,t_2,s_2, \ldots$
determined, in that order, by
\[
    t_1 = \max \{t : at \in \Gamma \} ,  \;\;\;\; s_1 = a,
\]
\[
    t_{i+1} = \max \{ t: \exists s < t_i \mbox{ such that }
    st \in \Gamma  \}, \;\;\;\;
    s_{i+1} = \min \{ s : st_{i+1} \in \Gamma \}  .
\]
The procedure terminates when $t_{i+1}=b$.
Given a lace $L$, the set of all edges $st \not\in L$ such that
${\sf L}_{L\cup \{st\} } = L $ is denoted  $\Ccal (L)$. Edges in
$\Ccal (L)$ are said to be \emph{compatible}
with $L$. Fig.~\ref{fig:graphs} illustrates these
definitions.
\end{defn}

\begin{figure}[t]
\begin{center}
	\includegraphics[width=10cm, height=8 cm]{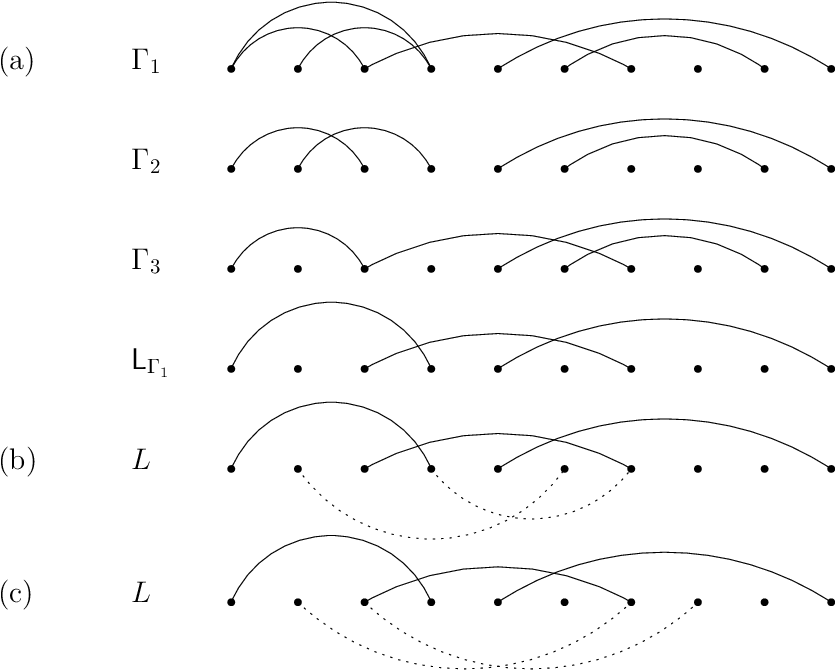}
\end{center}
\caption{ \label{fig:graphs}
(a) $\Gamma_1,\Gamma_2,\Gamma_3 $ are graphs on $[0,9]$. $\Gamma_1$ is connected while $\Gamma_2$ and $\Gamma_3$ are disconnected. ${\sf L}_{\Gamma_1}$ is the lace associated with $\Gamma_1$ in Definition~\ref{def:lace_presc}. (b) The dotted edges are compatible with the lace $L$. (c) The dotted edges are not compatible with the lace $L$.
}
\end{figure}

\subsection{Definition of $\Pi$}
\label{sec:Pi}

Suppose that to each
walk $\omega = (\omega(0), \omega(1), \ldots , \omega(n))$,
either
on $\Z^d$
or on $\T_r^d$,
and to
each pair $s,t \in \{ 0,1,\ldots ,n\}$ with $s<t$ we are given
a real number $\Ucal_{st}(\omega)$.
We are interested in the choice $\Ucal_{st} = \beta U_{st}$ with $U_{st}$ defined
in \eqref{e:Ustdef}.
Let $K [0,0] = 1$, and for $n>0$ let
\begin{align}
\label{e:Kdef}
    K [0,n] & = \prod_{0 \leq s<t \leq n}
    ( 1 + \Ucal_{st} ) ,
\\
\label{e:Jlacesum}
    J [0,n] &=
    \sum_{L \in \Lcal [0,n] } \prod_{st \in L} \Ucal_{st}
    \prod_{s't' \in \Ccal (L) } ( 1 + \Ucal_{s't'} ).
\end{align}
The dependence of $K[0,n]$ and $J[0,n]$ on the walk $\omega$ has been left implicit.
We define $J^{(N)}[0,n]$ to be the contribution to
\eqref{e:Jlacesum} from laces consisting of exactly $N$ bonds:
\begin{equation}
\label{e:JNdef}
    J^{(N)} [0,n] =
    \sum_{L \in \Lcal ^{(N)} [0,n] } \prod_{st \in L} \Ucal_{st}
    \prod_{s't' \in \Ccal (L) } ( 1 + \Ucal_{s't'} ) .
\end{equation}
Then
\begin{equation}
\label{e:Jdef}
    J [0,n] = \sum_{N=1}^{\infty} J^{(N)} [0,n].
\end{equation}

For $x \in \Z^d$, $N \ge 1$ and $n \ge 2$, let
\begin{align}
\label{e:PiNdef}
    \pi_n^{(N)} (x) & =  (-1)^N
    \sum_{\omega \in \Wcal_n( x)}
        J^{(N)} [0,n]
        \nnb
        &  =
        \sum_{\omega \in \Wcal_n( x)}
        \sum_{L \in \Lcal ^{(N)} [0,n] } \prod_{st \in L} (- \Ucal_{st})
    \prod_{s't' \in \Ccal (L) } ( 1 + \Ucal_{s't'} ).
\end{align}
The same definition applies for $x \in \T_r^d$ if we replace $\Wcal_n( x)$
by $\Wcal_n^\T( x)$.
The factor $(-1)^N$ on the right hand side of \eqref{e:PiNdef}
has been inserted so that
\begin{equation}
\label{e:PiNpos}
    \pi_m^{(N)}(x) \geq 0 \;\;\; \mbox{for all $N,n,x$}
\end{equation}
when $\Ucal_{st} \le 0$ for all $st$ as in our choice $\Ucal_{st}=\beta U_{st}$.
Then we define
\begin{equation}
\label{e:pidef}
    \pi_n (x) = \sum_{N=1}^\infty (-1)^N \pi_n^{(N)} (x)
    =
    \sum_{\omega \in \Wcal_n(0,x)}   J[0,n] .
\end{equation}
Let
\begin{equation}
\label{e:Pidef}
    \Pi_z (x) = \sum_{n=2}^\infty \pi_n(x) z^n = \sum_{N=1}^\infty (-1)^N \Pi_z^{(N)} (x),
\end{equation}
where
\begin{equation}
\label{e:PiNx}
    \Pi^{(N)}_z (x)
    =
    \sum_{n=2}^\infty \pi_{n}^{(N)} (x) z^m   .
\end{equation}

For $L \in \Lcal^{(N)}$, the product
$\prod_{st \in L} (- \Ucal_{st}) = \beta^N \prod_{st \in L} (- U_{st})$ in \eqref{e:PiNdef}
contains an explicit factor $\beta^N$, and the remaining factor
$\prod_{st \in L} (- U_{st})$ is either $0$ or $1$, with the value $1$ occurring if
and only if the walk $\omega$ obeys $\omega(s)=\omega(t)$ for each $s \neq t$.
Thus a walk $\omega$ contributing to $\pi^{(N)}_n(x)$ must have the self-intersections
depicted in Figure~\ref{fig:pi}.  These self-intersections divide $\omega$ into $2N-1$
subwalks.  The product over compatible edges in \eqref{e:PiNdef} enforces
weak self-avoidance within each subwalk, as well as providing mutual weak self-avoidance
between certain subwalks.

In particular, for $N=1$ and $\Ucal_{st}=\beta U_{st}$,
we have $\Pi_z^{(1)}(x)=0$ if $x \neq 0$ and also
\begin{equation}
    \Pi^{(1)}_z (0)
    =
        \sum_{n=2}^\infty z^n \sum_{\omega \in \Wcal_n( 0)}
         (- \Ucal_{0n})
    \prod_{s't' \in \Ccal (L) } ( 1 + \Ucal_{s't'} ).
\end{equation}
The factor $- \Ucal_{0n}$ is equal to $\beta$, and the product over compatible edges
would be completed to a product over all edges by the inclusion of a factor
$1+\Ucal_{0n}$, which is $1-\beta$ when $\omega(n)=0$.  This completion would be the relevant product for $G_z(0)$ (i.e., $K[0,n]$).
%but the term $n=0$ is not included in $\Pi^{(1)}_z (0)$.
From this, we see that,
for $\beta \in [0,1)$,
\begin{equation}
\label{e:Pi1-bis}
    \Pi^{(1)}_z (0) = \frac{\beta}{1-\beta}(G_z(0)-1).
\end{equation}

\begin{figure}[h!]
\begin{center}
	\includegraphics[width=10cm, height=5 cm]{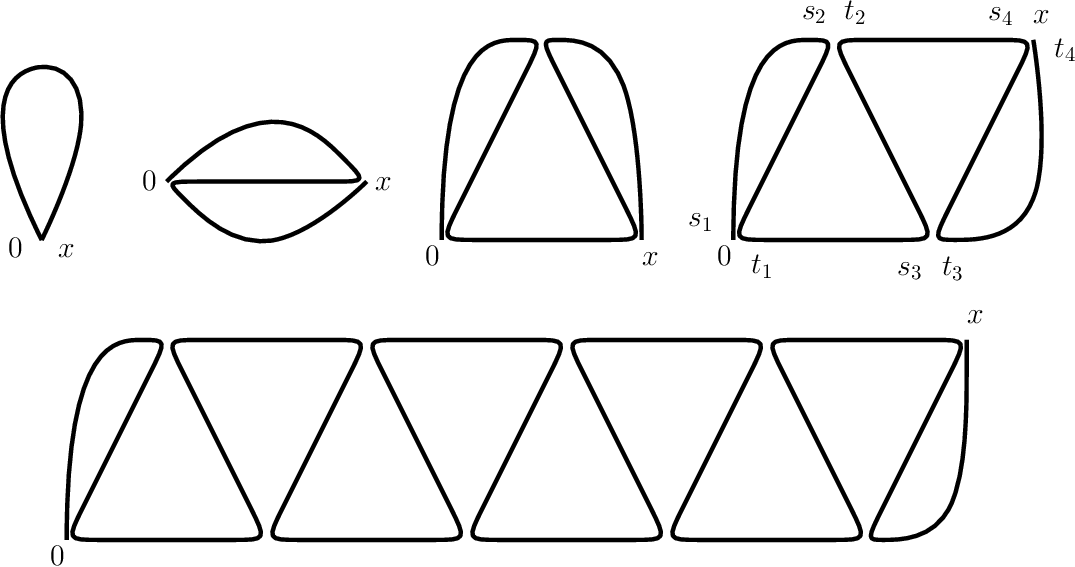}
\end{center}
\caption{Topology of walks contributing to $\pi^{(N)}_m(x)$ for $N=1,2,3,4$ and $10$.}
\label{fig:pi}
\end{figure}

By definition, with the choice $\Ucal_{st}=\beta U_{st}$,
\begin{equation}
    c_n(x) = \sum_{\omega \in \Wcal_n(x)} K[0,n],
    \qquad
    G_z(x) = \sum_{n=0}^\infty c_n(x) z^n ,
\end{equation}
and the same equations hold for $c_n^\T(x)$ and $G_z^\T(x)$ when $\Wcal_n(x)$ is replaced
by $\Wcal_n^\T(x)$.
The following proposition is a statement of the lace expansion.
It was originally proved in \cite{BS85}, see also \cite[(3.14)]{Slad06}.
Although those references are for $\Z^d$, the proof of the proposition applies verbatim
to the torus.

\begin{prop}
\label{prop:lace}
For $n \ge 1$ and for $x \in \Z^d$,
\begin{equation}
    c_n(x) = (2dD * c_{n-1})(x) + \sum_{m=2}^n (\pi_m * c_{n-m})(x),
\end{equation}
and hence
\begin{equation}
    G_z(x) = \delta_{0,x} + (2dD *G_z)(x) + (\Pi_z * G_z)(x).
\end{equation}
The same holds for $x \in \T_r^d$ with $c_n$ and $G_z$ replaced by $c_n^\T$ and $G_z^\T$,
with $\pi_n$ and $\Pi_z$ defined via walks in $\Wcal^\T_n(x)$, and with the torus convolution
$\star$ in place of $*$.
\end{prop}

\subsection{Diagrammatic estimates}

We define the multiplication and convolution operators
\begin{align}
\label{e:Mcal}
    (\Mcal_z f)(x) & =  G_z(x)f(x),
    \\
\label{e:Gcal}
    (\Gcal_z f)(x) & =  (G_z *f)(x),
\end{align}
for $f: \Z^d \to \R$ and $x \in \Z^d$.
A proof of the
diagrammatic estimate \eqref{e:PiNconv} can be found at \cite[(4.40)]{Slad06},
and the identity \eqref{e:Pi1} is derived above as \eqref{e:Pi1-bis}.

\begin{prop}
For $z \ge 0$,
\begin{equation}
\label{e:Pi1}
    \Pi^{(1)}_z(x)
    =
    \delta_{0,x} \frac{\beta }{1-\beta} (G_z(0) -1),
\end{equation}
and, for $N \ge 2$,
\begin{equation}
\label{e:PiNconv}
    \sum_{x\in \Z^d} \Pi^{(N)}_z(x)
    \le
    \beta^N
    \left[ (\Gcal_z \Mcal_z)^{N-1}  G_z\right] (0).
\end{equation}
\end{prop}
The following lemma (see \cite[Lemma~4.6]{Slad06})
gives a way to bound the right-hand side of \eqref{e:PiNconv}.

\begin{lemma}
\label{lem:convol_bound}
Given nonnegative
even functions $f_0,f_1,\ldots,f_{2M}$ on $\Z^d$, for $j=1,\ldots,M$ let $\Gcal_{j}$
and $\Mcal_{j}$ be respectively the operations of
convolution with $f_{2j}$ and
multiplication by $f_{2j-1}$.  Then for any
$k \in \{0,\ldots, 2M\}$,
\begin{equation}
\label{e:HMjboundaz}
    \|\Gcal_{M}\Mcal_{M}\cdots \Gcal_{1}\Mcal_{1} f_0\|_\infty
    \leq \|f_k\|_\infty \prod  \|f_{j}*f_{j'}\|_\infty ,
\end{equation}
where the product is over disjoint consecutive pairs
$jj'$ taken from the set $\{0,\ldots, 2M\}\setminus \{k\}$
(e.g., for $k=3$ and $M=3$, the product has factors with
$jj'$ equal to $01$, $24$, $56$).
The same holds for functions on the torus, with convolutions replaced by the torus convolution $\star$.
\end{lemma}

By \eqref{e:PiNconv} and Lemma~\ref{lem:convol_bound}, for $N \ge 2$ we find that
for $z \in [0, z_c]$,
\begin{equation}
\label{e:PiNbub}
    \sum_{x\in \Z^d} \Pi^{(N)}_z(x)
    \le
    \beta^N \|G_z\|_\infty
    \|G_z * G_z\|_\infty^{N-1}
    ,
\end{equation}
and similarly for the torus.
The following proposition gives an estimate on the $z$-derivative of $\Pi_z$
(see \cite[(5.4.19)]{MS93} for an analogous statement).

\begin{prop}
\label{prop:Pidots}
Let $z \ge 0$.
For $N=1$ and $\beta \in (0,\frac 12]$,
\begin{equation}
\label{e:Pi1bd}
    \|  \partial_z   \Pi^{(1)}_z \|_1
    \le 2\beta \|\partial_z G_z\|_\infty,
\end{equation}
and for $N \ge 2$,
\begin{align}
\label{e:Pizder}
    \|  \partial_z    \Pi^{(N)}_z \|_1
    &\le
    (2N-1) \beta^N
    \|\partial_z  G_z\|_\infty
    \|G_z * G_z \|_\infty^{N-1}
    .
\end{align}
The same holds on the torus with convolution $\star$.
\end{prop}

\begin{proof}
The inequality \eqref{e:Pi1bd} is immediate since
$\Pi_z^{(1)}(x)= \delta_{0,x} \frac{\beta}{1-\beta}(G_z(0)-1)$ by \eqref{e:Pi1}.
For $N \ge 2$, by definition
\begin{equation}
    \partial_z  \Pi^{(N)}_z(x) = \sum_{n=2}^\infty n \pi_n^{(N)}(x) z^{n-1}.
\end{equation}
The diagrammatic representation of $\pi_n^{(N)}(x)$ has
$2N-1$ subwalks of total length $n$.  Let $n_i$ be the length of the $i^{\rm th}$
subwalk, so that
\begin{equation}
    n =  \sum_{i=1}^{2N-1} n_i.
\end{equation}
Use of this equality leads to a modification of \eqref{e:PiNconv} in which one
factor $\Gcal_z$ or $\Mcal_z$, or the factor $G_z$, has its function replaced by
$\partial_z G_z$.  Consequently, with \eqref{e:HMjboundaz}
and choosing the $i^{\rm th}$ line as the distinguished
line, we see that
\begin{equation}
\label{e:dzPibd}
    \partial_z  \Pi^{(N)}_z(x)
    \le
    (2N-1)  \beta^N \|\partial_z  G_z\|_\infty
    \|G_z * G_z \|_\infty^{N-1}
\end{equation}
 and the proof is complete.
\end{proof}

\section{Proof of Propositions~\ref{prop:Flacebds-new}--\ref{prop:Flbd-new}:
Susceptibility estimates}
\label{sec:lacepfs}

We now prove Propositions~\ref{prop:Flacebds-new}--\ref{prop:Flbd-new}.

\subsection{Proof of Proposition~\ref{prop:Flacebds-new}}
\label{sec:Flacebds}

For convenience we repeat the statement of Proposition~\ref{prop:Flacebds-new} here
as follows.

\begin{prop}[same as Proposition 2.1]
\label{prop:Flacebds-bis}
Let $d>4$ and let $\beta$ be sufficiently small.
The derivative $F'(z)$ obeys $F'(z)=-2d+O(\beta)$ for all $|z|\le z_c$, and $-z_cF'(z_c)=1+O(\beta)$.  Also, $F$ obeys the lower bound
\begin{align}
\label{e:Fder-bis}
    |F(z)| & \succ   |1-z/z_c|  \qquad (|z| \le z_c)
    .
\end{align}
\end{prop}

\begin{proof}
To abbreviate the notation, we write
\begin{equation}
\label{e:Pihatdef}
    \hat \Pi^{(N)}_z = \sum_{x\in\Z^d}\Pi^{(N)}_z(x),
    \qquad
    \hat \Pi_z = \sum_{x\in\Z^d}\Pi_z(x)
    = \sum_{N=1}^\infty (-1)^N \hat \Pi^{(N)}_z,
\end{equation}
with the hat notation consistent with a Fourier transform evaluated at $k=0$.
Since $|\hat\Pi^{(N)}_z| \le \hat\Pi^{(N)}_{|z|}$ and similarly for the derivative,
we can obtain bounds for complex $z$ from bounds with positive real $z$.

Using submultiplicativity we find that for $z \ge 0$
\begin{align}
    z\partial_z G_z(x)
    & =
    \sum_{n=1}^\infty nc_n(x) z^n
     = \sum_{n=1}^\infty \sum_{i=1}^n c_n(x) z^n
    \le \sum_{n=1}^\infty \sum_{i=1}^n (c_i*c_{n-i})(x) z^iz^{n-i}
    \nnb &
    =
    \sum_{y\in\Z^d} \sum_{i=1}^\infty c_i(x-y)z^i \sum_{k=0}^\infty c_k(y) z^k
    \le (G_z*G_z)(x)
    .
\label{e:nGG}
\end{align}
We combine \eqref{e:PiNbub} and Proposition~\ref{prop:Pidots}
(together with \eqref{e:nGG}) with
the fact that $\|G_z\|_\infty$ and $\|G_z*G_z\|_\infty$ are uniformly bounded
(recall \eqref{e:bubblebd})
to see that
there is a positive constant $C$ such that
\begin{align}
\label{e:bd_Pi_hat_N_and_prime}
    |\hat\Pi_z^{(N)}| & \le (C\beta)^N,
    \qquad
    |\partial_z \hat\Pi_z^{(N)}|  \le (C\beta)^N ,
\end{align}
uniformly in complex $z$ with $|z| \le z_c$.
The constant $C$ has been chosen large enough
to absorb a prefactor $2N-1$ in the derivative bound;
this manoeuver will be repeated later to avoid writing unimportant polynomial factors in $N$.
A small detail is that in the upper bound from  \eqref{e:dzPibd}
it is $\partial_z G_z$
that appears and not $z\partial_z G_z$ as in \eqref{e:nGG} and this could in principle
be problematic when $z$ gets close to $0$. But in fact this is unimportant because
for $z \le \frac{1}{4d}$ we can use
\begin{equation}
    \partial_z G_z(x) = \sum_{n=1}^\infty nc_n(x)z^{n-1}
    \le \sum_{n=1}^\infty n(2d)^n z^{n-1}
\end{equation}
which remains bounded for $z \in [0, \frac{1}{4d}]$.  For
$z \in [\frac{1}{4d},z_c]$ no difficulty is posed
by the occurrence of $\partial_zG_z$ rather than $z\partial_z G_z$ in the bound \eqref{e:nGG}.
Thus we can choose $\beta$ to be sufficiently small that $\hat\Pi_{z}$ and $\partial_z \hat\Pi_z$
are each $O(\beta)$ uniformly in the complex disk $|z| \le z_c$.

Since $F(z)=1-2dz-\hat\Pi_z$, the derivative is
\begin{equation}
    F'(z) = -2d -\partial_z \hat\Pi_z = -2d +O(\beta).
\end{equation}
Also, since $\chi(z_c)=\infty$, we have $F(z_c)=1-2dz_c-\hat\Pi_{z_c}=0$, which
implies that $2dz_c=1+O(\beta)$ and hence that $-z_cF'(z_c) = 2dz_c + z_c\partial_z \hat\Pi_{z_c} =1+O(\beta)$.
It also gives
\begin{equation}
    F(z) = F(z) - F(z_c) = 2d(z_c-z) +[\hat\Pi_{z_c}- \hat\Pi_z],
\end{equation}
which by the Fundamental Theorem of Calculus applied to $f(t)=\hat\Pi_{(1-t)z+tz_c}$ yields
\begin{equation}
    F(z) = (z_c-z)\Big[ 2d + \int_0^1 \partial_z \hat\Pi_z\big|_{(1-t)z+tz_c} \D t   \Big]  .
\end{equation}
The lower bound in the disk $|z|\le z_c$ follows from the $O(\beta)$
bound on the derivative of
$\hat\Pi_z$ in the disk.
\end{proof}

\subsection{Proof of Proposition~\ref{prop:Flbd-new}}
\label{sec:Flbd}

For convenience we restate Proposition~\ref{prop:Flbd-new} as follows.

\begin{prop}[same as Proposition~\ref{prop:Flbd-new}]
\label{prop:Flbd-bis}
Let $d>4$ and let $\beta$ be sufficiently small.  Let $\zeta=z_c(1-V^{-1/2})$ and
$U=\{z\in \C : |z|<\zeta\}$.
Then $\varphi'(z) = -2d +O(\beta)$ for all $|z|\le\zeta$,
and $\varphi$ obeys the lower bound
\begin{equation}
\label{e:varphilb-bis}
    |\varphi(z)| \succ
    |1-z/\zeta|
    \qquad (|z| \le \zeta)
    .
\end{equation}
\end{prop}

\begin{proof}
We first observe that
\begin{equation}
    \varphi(z) = 1-2dz - \hat\Pi^\T_z,
\end{equation}
from which we conclude that
\begin{equation}
    \varphi(z) = \varphi(\zeta) + 2d(\zeta-z) + [\hat\Pi^\T_\zeta - \hat\Pi^\T_z].
\end{equation}
As in the proof of Proposition~\ref{prop:Flacebds-bis}, we control $\hat \Pi_z^{\T,(N)}$ (and thus $\hat \Pi_z^{\T}$) using Proposition \ref{prop:Pidots}.
To make use of these diagrammatic estimates,
we first observe that $\chi(\zeta) \prec
V^{1/2}$ by \eqref{e:chiZdasy}, and hence by the upper bound of
Theorem~\ref{thm:plateau} (for the two-point function $G_z^\T$) and by \eqref{e:bubblebd},
\eqref{e:GGam} and
\eqref{e:GGstar} (for the torus convolution $G_z^\T\star G_z^\T$), together with the
uniform bounds on the $\Z^d$ two-point function $G_z$ and bubble $\bubble_z$ at $\zeta < z_c$, we find that the inequalities
\begin{align}
    &\|G_{z}^\T\|_\infty
    \prec \|G_z\|_\infty + \frac{ \chi(|z|) }{V}
    \prec 1 + \frac{V^{1/2}}{V}
    \prec 1, \\
\label{e:GG_norm_tor_bd}
   &\|G_{z}^\T \star G_{z}^\T\|_\infty
   \le \|\Gamma_{z}  \star \Gamma_{z} \|_\infty
   \prec 1+ \frac{\chi(|z|)^2}{V} \leq 1+ \frac{V}{V} \prec 1
\end{align}
hold uniformly in $z \in U$.
Consequently, by the torus versions of \eqref{e:PiNbub} and Proposition~\ref{prop:Pidots},
there is a constant $C$ (independent of $\beta,z,r$) such that, uniformly in $z\in U$,
\begin{align}
    |\hat\Pi_{z}^{\T,(N)}| & \le (C\beta)^N,
    \\
\label{e:PiTder}
    |\partial_z \hat\Pi_{z}^{\T,(N)}| & \le (C\beta)^N.
\end{align}
By summing the above two inequalities over $N\geq 1$, we see that both $\hat\Pi_z^{\T}$ and its $z$-derivative are $O(\beta)$ uniformly in $z \in U$.
In particular, this implies that $\varphi'(z)=-2d-\partial_z\hat\Pi_z^{\T} = -2d+O(\beta)$
as claimed.

Finally, for the lower bound on $|\varphi(z)|$ we
apply the Fundamental Theorem of
Calculus to the first order, to
the function $\hat\Pi^\T_\zeta - \hat\Pi^\T_z$,
to see that
\begin{align}
    \varphi(z)
    & =
    \varphi(\zeta) + 2d(\zeta-z) + O(\beta|\zeta-z|).
\end{align}
Since $z \in U$, we know that ${\rm Re}\,(\zeta-z ) \ge 0$ and hence, since $\varphi(\zeta)>0$,
\begin{align}
    \Big| \varphi(\zeta) + 2d(\zeta-z) + O(\beta|\zeta-z|) \Big|
    & \geq
    \Big| \varphi(\zeta) + 2d(\zeta-z)\Big| - O(\beta|\zeta-z|)
    \nnb & \succ
    |1-z/\zeta| - O(\beta |1-z/\zeta|) \succ |1-z/\zeta|
\end{align}
for $\beta$ small enough.
This completes the proof.
\end{proof}

\begin{rk}
\label{rk:beta-s}
Our restriction to $|z|\le \zeta = z_c(1-V^{-1/2})$ is present in order to achieve
$V^{-1}\chi(\zeta)^2 \prec 1$ in \eqref{e:GG_norm_tor_bd}.  The true limitation of our
method
is that $\beta \|G_\zeta^\T \star G_\zeta^\T\|_\infty$ must be sufficiently small
that the sum over $N$ converges and the sum remains small.  If we had instead
defined $\zeta = z_c(1-sV^{-1/2})$ with some small positive $s$, then
it would be necessary to take $\beta$ small depending on $s$.  The fact that
our method requires this, in spite of the fact that we believe that Conjecture~\ref{conj:window}
remains true for all real $s$, is an indication that a new idea is needed in order
to analyse $c_n^\T$ and $\chi^\T(z)$
(with fixed positive $\beta$)
for $n$ above $V^{1/2}$ or for $z$ of
the form $z = z_c(1-sV^{-1/2})$ for \emph{all} real $s$.
\end{rk}

\section{Proof of Propositions~\ref{prop:AD-new} and \ref{prop:AD-new-prime}: Susceptibility comparison}
\label{sec:AD}

In this section, we prove the bounds on  $\Delta$ and $\Delta'$
stated in Propositions~\ref{prop:AD-new} and \ref{prop:AD-new-prime}.
The bound on $\Delta$ is needed for all our main results, whereas the bound
on its derivative is needed only for the proof of Theorem~\ref{thm:EL} for the expected length.

\subsection{Start of proof}

Recall that $\Delta(z)$ is defined for $z\in \C$ with $|z|\le z_c$
by
\begin{equation}
    \Delta(z) = \varphi(z)-F(z)= \hat\Pi_z- \hat\Pi^\T_z,
\end{equation}
where as in \eqref{e:Pihatdef} we write
\begin{equation}
    \hat\Pi_z = \sum_{x\in \Z^d}\Pi_z(x),
    \qquad
    \hat\Pi^\T_z = \sum_{x\in \T_r^d}\Pi^\T_z(x).
\end{equation}
We also are interested in the derivative $\Delta'(z) = \varphi'(z)-F'(z)$.
Recall that the closed disk $U$ is defined by
\begin{equation}
\label{e:Udef}
    U = \{z \in \C :\, |z| \le  z_c(1-V^{-1/2}) \}.
\end{equation}
For convenience, we combine Propositions~\ref{prop:AD-new} and \ref{prop:AD-new-prime}
into the following proposition.

\begin{prop}[same as Propositions~\ref{prop:AD-new} and \ref{prop:AD-new-prime}]
\label{prop:AD-bis}
Let $d>4$ and let $\beta$ be sufficiently small.  For $z\in U$,
\begin{align}
\label{e:Deltabd}
    |\Delta(z)|
    & \prec \beta  \frac{r^2+\chi(|z|)}{V}  ,
    \qquad
    |z\Delta'(z)|
     \prec \beta   \frac{\chi(|z|)(r^2+\chi(|z|))}{V}
    .
\end{align}
\end{prop}

To prove Proposition~\ref{prop:AD-bis}, we must compare the $\Z^d$ and torus susceptibilities
as well as their derivatives.
This is equivalent to a comparison of
$\Pi_{z}$ and $\Pi_{z}^\T$,
as well as a comparison
of their derivatives.  No such direct comparison of torus and $\Z^d$ lace
expansions has been performed
previously in the literature, and the proof requires new ideas.  The torus plateau
upper bounds from Theorem~\ref{thm:plateau}, and its consequences
for the bubble and triangle in Lemmas~\ref{lem:G3}--\ref{lem:GT3}, are indispensable for this.

Let $\pi_r : \Z^d \to \T_r^d$ be the canonical projection onto the torus.
To begin the comparison, for a walk
$\omega$ on $\Z^d$ let
\begin{align}
\label{e:Utorus}
	U^\T_{st}(\omega) &=
    \begin{cases}
        -1 & (\pi_r\omega(s)=\pi_r\omega(t))
        \\
        0 & (\text{otherwise})
    \end{cases}
\end{align}
and for $\omega$ of length $n$ let $K^\T[0,n] = \prod_{0 \leq s < t \leq n}(1+\beta U_{st}^{\T})$.
Via the lift defined in \eqref{e:liftdef},
a torus walk to $x$ lifts bijectively to a $\Z^d$-walk ending at a point of the
form $x+ru$ for some $u \in \Z^d$.
We can therefore rewrite the torus two-point function
as a sum over walks on $\Z^d$, as
\begin{align}
\label{e:GTKT}
    G_z^\T(x)
    & =
    \sum_{n=0}^\infty z^n \sum_{\omega \in \Wcal^\T_n(x)}
    K[0,n]
     =
    \sum_{n=0}^\infty z^n \sum_{u \in \Z^d} \sum_{\omega \in \Wcal_n(x+ru)}
    K^\T[0,n]  \qquad (x \in \T_r^d),
\end{align}
where as usual on the right-hand side we identify
$x$ with a point in $[-r/2,r/2)^d \cap \Z^d$.
Similarly, by the definition of $\Pi^{\T,(N)}_z$ in \eqref{e:PiNx},
with $\Wcal_n^\T = \cup_{x\in \T_r^d} \Wcal_n^\T(x)$,
\begin{equation}
\label{e:PiTdef}
    \hat\Pi_z^{\T,(N)}
    =
     \sum_{n=2}^\infty z^n \sum_{\omega \in \Wcal_n^\T }
        \sum_{L \in \Lcal ^{(N)} [0,n] } \prod_{st \in L} (- \beta U_{st})
    \prod_{s't' \in \Ccal (L) } ( 1 + \beta U_{s't'} )
     .
\end{equation}
Because of the one-to-one correspondence between torus and $\Z^d$ walks, there is
an equivalent formulation involving walks on $\Z^d$ rather than on the torus, namely
\begin{align}
\label{e:PiTdef-Zd}
    \hat\Pi_z^{\T,(N)}
    &=
     \sum_{n=2}^\infty z^n
     \sum_{\omega \in \Wcal_n}
        \sum_{L \in \Lcal ^{(N)} [0,n] } \prod_{st \in L} (-\beta U^\T_{st})
    \prod_{s't' \in \Ccal (L) } ( 1 + \beta U^\T_{s't'} )
    \nnb &
    =
    \sum_{n=2}^\infty z^n \sum_{\omega \in \Wcal_n} J^{\T,(N)}[0,n]
     ,
\end{align}
where the last equality defines $J^{\T,(N)}[0,n]$.
In contrast to \eqref{e:PiTdef} which involves walks on the torus with $U_{st}$
taking effect for
intersections of the torus walk,
the formula \eqref{e:PiTdef-Zd} involves walks on
$\Z^d$ with
the interaction $U^\T_{st}$ taking
effect when the walk visits points with the same torus projection.

By definition,
\begin{align}
\label{e:DeltaT}
    \Delta (z) & =
    \sum_{N=1}^\infty (-1)^{N+1}\Delta^{(N)}(z),
\end{align}
where $\Delta^{(N)} (z)$ is defined by
\begin{equation}
\label{e:def-Delta_N}
	\Delta^{(N)} (z) = \hat\Pi^{\T,(N)}_z - \hat\Pi^{(N)}_z.
\end{equation}
The use of \eqref{e:PiTdef-Zd} rather than \eqref{e:PiTdef} facilitates  the comparison
on $\Z^d$ and the torus, as required to prove Proposition~\ref{prop:AD-bis}.
Indeed, by \eqref{e:def-Delta_N} and \eqref{e:PiNdef},
\begin{align}
\label{e:Delta_N_reform}
	\Delta^{(N)}(z)
    &=
    \sum_{n=2}^\infty z^n \sum_{w\in \Wcal_n}(J^{\T,(N)}[0,n]-J^{(N)}[0,n]).
\end{align}
The following proposition is sufficient to prove Proposition~\ref{prop:AD-bis}.

\begin{prop}
\label{prop:Deltabds}
Let $d>4$ and let $\beta$ be sufficiently small.
Let $z\in U$ and $N \ge 1$.
There exists a positive constant $C$ independent of $\beta,N,r$ and $z$
such that
\begin{align}
    |\Delta^{(N)}(z)| & \leq
    (C\beta)^N \frac{r^2+\chi(|z|)}{V},
    \qquad
    |z\Delta'^{(N)}(z)| \leq
    (C\beta)^N \frac{\chi(|z|)(r^2+\chi(|z|))}{V}.
\end{align}
\end{prop}

\begin{proof}[Proof of Proposition~\ref{prop:AD-bis}]
By choosing $\beta$ small enough
to give convergence of the geometric series in \eqref{e:DeltaT}, we obtain
\begin{equation}
    |\Delta(z)|
    \prec \beta \frac{r^2+\chi(|z|)}{V}
    ,
    \qquad
    |z\Delta'(z)|
    \prec \beta \frac{\chi(|z|)(r^2+\chi(|z|))}{V}
    ,
\end{equation}
as desired.
\end{proof}

\begin{rk}
\label{rk:zeta}
The restriction that $z \in U$ ensures that $V^{-1}\chi(|z|)^2$ is at most of order $O(1)$. This observation will be used repeatedly in what follows to disregard factors of the form $(1+V^{-1}\chi(|z|)^2)$.
\end{rk}

For convenient reference, we assemble the following definitions here.
For a  walk $\omega$ in $\Z^d$ and any edge $st$ with $s<t\leq|\omega|$, we recall
the definitions of $U_{st},U^\T_{st}$
from \eqref{e:Ustdef} and \eqref{e:Utorus},
and also define $U^+_{st}$ as follows:
\begin{align}
\label{e:def-U_T_and_+}
	U_{st}(w) &=
    \begin{cases}
        -1 & (\omega(s)=\omega(t))
        \\
        0 & (\text{otherwise}),
    \end{cases}
\\
	U^\T_{st}(w) &=
    \begin{cases}
        -1 & (\pi_r\omega(s)=\pi_r\omega(t))
        \\
        0 & (\text{otherwise}),
    \end{cases}
\\
\label{e:Uplusdef}
	U^+_{st}(w) &=     \begin{cases}
        -1 & (\pi_r\omega(s)=\pi_r\omega(t)\text{ and } \omega(s)\neq \omega(t))
        \\
        0 & (\text{otherwise}).
    \end{cases}
\end{align}
By definition,
\begin{equation}
\label{e:UT+}
	(1+\beta U^\T_{st}) = (1+\beta U_{st})(1+\beta U^+_{st}).
\end{equation}
With $K^\#[0,n] = \prod_{0 \le s < t \le n}(1+\beta U^{\#}_{st})$ for $\#$ any of $\T, +$, or nothing,
we therefore have
\begin{equation}
\label{e:KKK}
	K^\T[0,n] = K[0,n]  K^+[0,n] .
\end{equation}

\subsection{$1$-loop diagram}

We first prove the case $N=1$ of Proposition~\ref{prop:Deltabds}.
By \eqref{e:Pi1},
\begin{equation}
	\hat\Pi^{(1)}_z  = \frac{\beta}{1-\beta}(G_z(0)-1).
\end{equation}
It is the same for
$\hat\Pi^{\T,(1)}_z$ and thus
\begin{equation}
\label{e:Delta_1}
	\Delta^{(1)}(z) = \frac{\beta}{1-\beta}(G^\T_z(0)-G_z(0)).
\end{equation}

\begin{prop}
\label{prop:Delta_1_bound}
Let $d>4$ and let $\beta$ be sufficiently small.
There is a constant $C$ independent of $\beta,r, z$ such that for all $z \in \C$
with $|z|\le z_c$
\begin{equation}
	 |\Delta^{(1)}(z)| \le C\beta \frac{\chi(|z|)}{V}.
\end{equation}
\end{prop}

\begin{proof}
We use \eqref{e:Delta_1} to see that it is enough to control $G_z^\T(0) - G_z(0)$,
which by \eqref{e:GTKT} and \eqref{e:KKK} is equal to
\begin{align}
	G_z^\T(0) - G_z(0)
    & =
    \sum_{n=0}^\infty z^n \sum_{u \in \Z^d} \sum_{\omega \in \Wcal_n(ru)}   K^\T[0,n]
    -
    \sum_{n=0}^\infty z^n \sum_{\omega \in \Wcal_n(0)}   K[0,n]
    \nnb & =
    \sum_{u \neq 0} \sum_{n=0}^\infty z^n \sum_{\omega \in \Wcal_n(ru)}   K^\T[0,n]
    -
    \sum_{n=0}^\infty z^n \sum_{\omega \in \Wcal_n(0)}   K[0,n](1-K^+[0,n]).
\label{e:Delta1pf}
\end{align}
The absolute value of the first term can be bounded using the inequality $K^\T[0,n] \leq K[0,n]$
together with \eqref{e:plateau-ub} (recall the definition
of $\Gamma$ in \eqref{e:def-gamma}) as
\begin{align}
\label{e:GGamG}
     \sum_{u \neq 0}G_{|z|}(ru)
     =
     \Gamma_{|z|}(0)-G_{|z|}(0)
     \prec \frac{\chi(|z|)}{V}
\end{align}
which is sufficient.
To bound $1-K^{+}[0,n]$,
we use the fact that for any discrete set $A$ and any choice of $u_a \in [0,1]$,
\begin{equation}
\label{e:bd_prod_to_sum}
	1 - \prod_{a\in A}(1-u_a) \leq \sum_{a\in A}u_a .
\end{equation}
The absolute value of the last term in \eqref{e:Delta1pf} is then bounded by
\begin{align}
    \beta\sum_{n=0}^\infty |z|^n \sum_{0 \le s < t \le n}
    \sum_{\omega \in \Wcal_n(0)}
    K[0,n] |U_{st}^+|.
\end{align}
For a nonzero contribution, the factor $U_{st}^+$ forces $\omega$ to visit two distinct
points $\omega(s)=y$ and $\omega(t)=y+ru$ with the same torus projection.
Thus, by relaxing the interaction between the three subwalks corresponding
to the intervals $[0,s]$, $[s,t]$, $[t,n]$, the above sum is bounded above
by
\begin{align}
    \beta\sum_{y\in\Z^d}\sum_{u \neq 0} G_{|z|}(y)G_{|z|}(ru)G_{|z|}(y+ru)
    & =
    \beta\sum_{u \neq 0} G_{|z|}(ru) \bubble_{|z|}(ru).
\end{align}
By Lemma~\ref{lem:G3} the bubble $\bubble_{|z|}(ru)$ is uniformly bounded,
and, as observed in \eqref{e:GGamG},
$\sum_{u \neq 0} G_{|z|}(ru)\prec \chi(|z|)/V$.  This completes the proof.
\end{proof}

Before considering the derivative $\Delta'^{(1)}(z)$, we first prove the following lemma.

\begin{lemma}
\label{lem:S_prime_case_1_as_psi}
Let $d>4$ and $\beta$ sufficiently small. Then for all $z \in [0,z_c]$,
	\begin{equation}
\label{e:S_prime_case_1_as_psi}
	\sup_{x \in \Z^d}\sum_{u \neq 0}(G_z(ru){\sf T}_z(x+ru)+\bubble_z(ru)\bubble_z(x+ru))
    \prec \frac{\chi(z)^2}{V}.
\end{equation}
\end{lemma}
\begin{proof}[Proof of Lemma \ref{lem:S_prime_case_1_as_psi}]
We estimate crudely as follows.
By \eqref{e:GGG}, ${\sf T}_z (x+ru) \prec m(z)^{-2} \prec \chi(z)$.
Also, since $\bubble_z(x+ru)\prec 1$, and with \eqref{e:plateau-ub} and
\eqref{e:Bbd}, we see that
\begin{align}
    \sup_{x \in \Z^d}\sum_{u \neq 0}(G_z(ru){\sf T}_z(x+ru)+\bubble_z(ru)\bubble_z(x+ru))
    & \prec
    \chi(z)\sum_{u\neq 0}G_z(ru) +\sum_{u\neq 0}\bubble_z(ru)
     \prec \frac{\chi(z)^2}{V}.
\end{align}
This completes the proof.
\end{proof}

\begin{prop}
\label{prop:Delta_1_prime_bound-upper}
Let $d>4$ and let $\beta$ be sufficiently small.
There is a constant $C$ independent of $\beta,r,z$ such that for all $z\in\C$
with $|z|\le z_c$,
\begin{equation}
    |z\Delta'^{(1)}(z)| \leq C\beta \frac{\chi(|z|)(r^2+\chi(|z|))}{V}.
\end{equation}
\end{prop}

\begin{proof}
From \eqref{e:Delta_1} we see that we need to estimate
$zG_z^{'\T}(0) - zG_z'(0)$,  which by differentiation of \eqref{e:Delta1pf} is
\begin{align}
    zG_z^{'\T}(0) - zG_z'(0)
	 & =
     \sum_{v \neq 0}\sum_{n=2}^\infty nz^n\sum_{\omega \in \Wcal_n(rv)}K^\T[0,n]
    - \sum_{n=2}^\infty nz^n\sum_{\omega \in \Wcal_n(0)}K[0,n](1-K^{+}[0,n])
    .
\label{e:GTonZd}
\end{align}
Since we use absolute bounds, we restrict in the rest of the proof to
real $z \in [0,z_c]$.

For the first term on the right-hand side of
\eqref{e:GTonZd}, we use the crude bound $K^\T \leq K$ and then
\eqref{e:nGG} to see that it is bounded above by
\begin{equation}
\label{e:Delta_prime_one_proof_1}
    \sum_{v \neq 0}z\partial_z G_z(rv) \le (G_z*G_z)(rv)=
    \sum_{x \in \T_r^d} \sum_{u \in \Z^d}\sum_{v\neq 0} G_{z}(x+ru)G_{z}(x+r(u-v)),
\end{equation}
where we have written the summation index implied by the convolution as
$ x+ru$ with $x \in \T_r^d$ and $u \in \Z^d$.
We consider separately the cases $u = 0$ and $u \neq 0$. For $u = 0$, we obtain
\begin{align}
	\sum_{x \in \T_r^d}\sum_{v\neq 0} G_{z}(x)G_{z}(x-rv)
	&= \sum_{x \in \T_r^d}G_z(x)(\Gamma_z(x)-G_z(x))
    \prec \frac{\chi(z)}{V}\sum_{x \in \T_r^d}G_z(x) \leq  \frac{\chi(z)^2}{V},
\end{align}
where we used \eqref{e:plateau-ub} for the first inequality. For $u \neq 0$, we obtain
\begin{align}
\sum_{x \in \T_r^d} \sum_{u \neq 0}\sum_{v\neq 0} G_{z}(x+r\chem{v})G_{z}(x+r(u-v))
    &\le
    \sum_{x \in \T_r^d} \sum_{u \neq 0} G_{z}(x+ur)
    \Big( G_{z}(x) + \sum_{w\neq 0}G_{z}(x+wr) \Big)
    \nnb
    &=
    \sum_{x \in \T_r^d} (\Gamma_{z}(x)- G_z(x)) \Gamma_z(x)
    \nnb &
    \prec
    \sum_{x \in \T_r^d}
    \frac{\chi(z)}{V}
    \Big(\frac{1}{r^{d-2}} +  \frac{\chi(z)}{V}\Big)
    \nnb & =
    \frac{\chi(z)}{V} ( r^2 +\chi(z)),
\end{align}
where we used \eqref{e:plateau-ub} in the third line.
This gives the desired estimate for the first term on the right-hand side of
\eqref{e:GTonZd}.

For the second sum in \eqref{e:GTonZd} we define $\widetilde \psi$ by
\begin{equation}
	\widetilde \psi(z)
    = \sum_{n=0}^\infty nz^{n} \sum_{\omega \in \Wcal_n(0)} K[0,n](1- K^{+}[0,n]).
\end{equation}
For the factor $1-K^{+}[0,n]$ we use \eqref{e:bd_prod_to_sum}.
Also, we write the factor $n$ as $n=\sum_{v\in\Z^d} \sum_{k=1}^n \1_{w(k) = v}$.
This leads to
\begin{align}
	\widetilde \psi (z)
	&\leq \beta
    \sum_{v \in \Z^d}\sum_{n = 0}^\infty\sum_{0\leq s<t \leq n}\sum_{k=1}^n z^n
    \sum_{\omega \in \Wcal_n(0)}|U^+_{st}|\,K[0,n]\1_{w(k) = v}.
\end{align}
By splitting $\omega$ into $4$ subwalks between the time intervals
separated by $s,t,k$, by neglecting interactions between these subwalks,
and by conditioning on $\omega(s)=w$ and $\omega(t) = w+ru$ ($u \neq 0$), we obtain
\begin{align}
\label{e:widetilde_psi}
	\widetilde{\psi}(z)
	&\leq \beta\sum_{u \neq 0}\sum_{v,w \in \Z^d}\Big(2G_{z}(v)G_{z}(w-v)G_{z}(ru)G_{z}(w+ru)
    \nnb
	&\qquad \qquad \qquad +\; G_{z}(w)G_{z}(v-w)G_{z}(w+ru-v)G_{z}(w+ru)\Big)
    \nnb
	&=\beta\sum_{u \neq 0}(2G_{z}(ru){\sf T}_{z}(ru)+\bubble_{z}(ru)^2) .
\end{align}
The first and second terms come respectively from the first and second diagrams in
Figure~\ref{fig:one_loop_prime} with $x=0$.
Since $k$ belongs to $[0,s]$, $[s,t]$ or $[t,n]$ there are in principle three different diagrams but two of them are identical by symmetry, hence the factor $2$ inside \eqref{e:widetilde_psi}.
Lemma~\ref{lem:S_prime_case_1_as_psi} provides
the required upper bound $V^{-1}\chi(z)^2$ for \eqref{e:widetilde_psi}.
This completes the proof.
\end{proof}

\begin{figure}[t]
\centering
	\includegraphics[scale = 0.9]{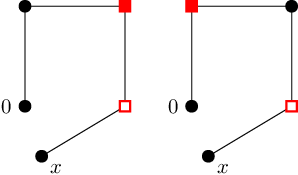}
	\caption{
Diagrammatic representation of the bound used on $\widetilde
\psi(p)$. There are three diagrams for the three possible
configurations with two of the diagrams equal by symmetry.
Each diagram is a path from $0$ to $x$
with three extra vertices,
two of which (square and box)
have the same torus projection but are distinct.
For \eqref{e:widetilde_psi} we are interested only in the case $x=0$. }
	\label{fig:one_loop_prime}
\end{figure}

\subsection{Higher-loop diagrams}

To deal with $N \ge 2$, we define $P^{(N)}[0,n]$ and $Q^{(N)} [0,n]$ by
\begin{align}
\label{e:def-P_def}
	P^{(N)}[0,n] &=
	\sum_{L \in \Lcal^{(N)}[0,n]}\prod_{st \in L}(-U_{st})\Big(\prod_{s't' \in \Ccal(L)}(1+\beta U_{s't'})-\prod_{s't' \in \Ccal(L)}(1+\beta U^\T_{s't'})\Big),
\\
\label{e:def-Q_def}
	Q^{(N)}[0,n] &=
	\sum_{L \in \Lcal^{(N)}[0,n]}\Big(\prod_{st \in L}(-U^\T_{st})- \prod_{st \in L}(-U_{st})\Big)\prod_{s't' \in \Ccal(L)}(1+\beta U^\T_{s't'}).
\end{align}
By the formula for $\Delta^{(N)}(z)$ in \eqref{e:def-Delta_N},
\begin{equation}
	\Delta^{(N)}(z)
    =
    \beta^N \sum_{n=2}^\infty z^n
    \sum_{\omega \in \Wcal_n}\Big(Q^{(N)}[0,n]-P^{(N)}[0,n]\Big).
\end{equation}
It follows from  $U_{st}^\T \leq U_{st} \leq 0$
that both $P^{(N)}[0,n]$ and $Q^{(N)}[0,n]$ are nonnegative.
Then we define
\begin{align}
	S^{(N)}(z) &= \beta^N\sum_{n = 2}^\infty z^n \sum_{\omega \in \Wcal_n}P^{(N)}[0,n],\\
	T^{(N)}(z) &= \beta^N\sum_{n = 2}^\infty z^n \sum_{\omega \in \Wcal_n}Q^{(N)}[0,n],
\label{e:TNdef}
\end{align}
so that
\begin{align}
\label{e:DeltaNTS}
	\Delta^{(N)}(z) &= T^{(N)}(z)-S^{(N)}(z).
\end{align}

Similarly,
\begin{align}
\label{e:DeltaT_prime}
    \Delta'(z)
    &= \sum_{N=1}^\infty (-1)^{N+1}\Delta'^{(N)}(z),
\end{align}
with
\begin{align}
\label{e:DeltaNTS_prime}
	\Delta'^{(N)}(z) &= T'^{(N)}(z)-S'^{(N)}(z),
\end{align}
where
\begin{align}
	zS'^{(N)}(z) &= \beta^N\sum_{n = 2}^\infty nz^n \sum_{\omega \in \Wcal_n}P^{(N)}[0,n],\\
	zT'^{(N)}(z) &= \beta^N\sum_{n = 2}^\infty nz^n \sum_{\omega \in \Wcal_n}Q^{(N)}[0,n].
\end{align}

We will prove the following proposition.
For simplicity, we do not attempt to prove sharp bounds on $S^{N}$ and $S'^{(N)}$,
although we believe that these terms are in fact smaller than $T^{N}$ and $T'^{(N)}$
(via improved decay in $V$, so not only smaller by an extra factor $\beta$).
We again absorb any polynomial factors in $N$ into $C^N$ as discussed below
\eqref{e:bd_Pi_hat_N_and_prime}.

\begin{prop}
\label{prop:bd_S_T-p}
Let $d>4$ and let $\beta >0$ be sufficiently small.
Let $z\in U$ and let $N \geq 2$.
There is a constant $C$ independent of $\beta,r,N,z$ such that
\begin{align}
\label{e:SNbd}
	|S^{(N)}(z)| &\leq \beta (C \beta)^{N}\frac{r^2+\chi(|z|)}{V},
\\
\label{e:SNpbd}
	|z S'^{(N)}(z)| &\leq  \beta(C\beta)^{N} \frac{\chi(|z|)(r^2+\chi(|z|))}{V},
\\
\label{e:TNbd}
	|T^{(N)}(z)| &\leq
    (C\beta)^{N} \frac{r^2+\chi(|z|)}{V},
\\
\label{e:TNpbd}
	|z T'^{(N)}(z)| &\leq (C\beta)^{N}
    \frac{\chi(|z|)(r^2+\chi(|z|))}
    {V}
    .
\end{align}
\end{prop}

\begin{proof}[Proof of Proposition~\ref{prop:Deltabds}]
For $N \geq 2$, this is a direct consequence of Proposition~\ref{prop:bd_S_T-p} together with \eqref{e:DeltaNTS} and \eqref{e:DeltaNTS_prime}. For $N=1$, it is an
immediate corollary of Propositions~\ref{prop:Delta_1_bound} and \ref{prop:Delta_1_prime_bound-upper}.
\end{proof}

It remains only to prove Proposition~\ref{prop:bd_S_T-p}.

\subsection{Proof of Proposition~\ref{prop:bd_S_T-p}}

We begin with some initial bounds on $S^{(N)}$ and $T^{(N)}$ in
Lemma~\ref{lem:Delta_as_T_m_S}.
Although we need estimates for complex $z$, we see from \eqref{e:DeltaNTS} and from non-negativity of $P^{(N)}[0,n]$ and $Q^{(N)}[0,n]$ that for any complex $z$,
\begin{equation}
    |\Delta^{(N)}(z)|
    \le
    S^{(N)}(|z|) + T^{(N)}(|z|),
\end{equation}
and similarly for $|\Delta^{'(N)}(z)|$.
For this reason, we will only be working with \emph{real non-negative $z$} in the rest of this section.

\subsubsection{Bounds on $S,T$}

\begin{lemma}
\label{lem:Delta_as_T_m_S}
For any $z \ge 0$ and $N \geq 2$,
$T^{(N)}(z)$ and $S^{(N)}(z)$ are nonnegative and satisfy
\begin{align}
\label{e:bd_first_S}
	S^{(N)}(z) &\leq
    \beta^{N+1}\sum_{n=1}^\infty z^n \sum_{\omega \in\Wcal_n} \sum_{L \in \Lcal^{(N)}[0,n]}
    \sum_{ab \in \Ccal(L)}(-U^+_{ab})\prod_{st \in L}(-U_{st})
    \prod_{s't' \in \Ccal(L)}(1+\beta U_{s't'}) ,
    \\
\label{e:bd_first_T}
	T^{(N)}(z) &\leq
    \beta^N\sum_{n=1}^\infty z^n \sum_{\omega \in\Wcal_n}\sum_{L \in \Lcal^{(N)}[0,n]}
    \sum_{ab \in L} (-U^+_{ab})\prod_{st\in L\setminus\{ab\}}(-U^\T_{st})
    \prod_{s't' \in \Ccal(L)} (1+\beta U^\T_{s't'}) .
\end{align}
In addition, $zS'^{(N)}(z)$ and $zT'^{(N)}(z)$ are bounded by the above right-hand sides
with $z^n$ replaced by $nz^n$.
\end{lemma}

\begin{proof}
Starting with $S^{(N)}(z)$, we note that it follows from \eqref{e:UT+} that
\begin{align}
	P^{(N)}[0,n] &=
    \sum_{L \in \Lcal^{(N)}[0,n]}\prod_{st \in L}(-U_{st})
    \prod_{s't' \in \Ccal(L)}(1+\beta U_{s't'})
    \Big(1-\prod_{s't' \in \Ccal(L)}(1+\beta U^+_{s't'})\Big).
\end{align}
We then use \eqref{e:bd_prod_to_sum}
to see that
\begin{align}
	P^{(N)}[0,n] &\leq
    \sum_{L \in \Lcal^{(N)}[0,n]}\prod_{st \in L}(-U_{st})
    \prod_{s't' \in \Ccal(L)}(1+\beta U_{s't'})
    \sum_{ab \in \Ccal(L)}(-\beta U^+_{ab}).
\end{align}
Upon multiplying by $z^n$ and summing over $n$ this gives the bound on $S^{(N)}$.

For $T^{(N)}$, we use the identity $U_{st}=U_{st}^\T(1+U_{st}^+)$ which is a
consequence of the definitions  \eqref{e:def-U_T_and_+}--\eqref{e:Uplusdef},
followed by \eqref{e:bd_prod_to_sum},
followed by the identity $U^+_{ab} U^\T_{ab} = U_{ab}^+$, to see that
\begin{align}
	\prod_{st \in L}(-U^\T_{st})- \prod_{st \in L}(-U_{st})
    &=
    \left[ 1 - \prod_{st \in L}(1 + U^{+}_{st}) \right]
    \prod_{st \in L}(-U^\T_{st})
    \nnb & \le
    \sum_{ab \in L} (-U_{ab}^+)
    \prod_{st \in L}(-U^\T_{st})
    \nnb & = \sum_{ab \in L}(-U^+_{ab})\prod_{st \in L\setminus\{ab\}}(-U^\T_{st}).
\end{align}
With \eqref{e:def-Q_def} and \eqref{e:TNdef}, this leads to \eqref{e:bd_first_T}.

For the derivatives it is the same, apart from the occurrence
of $nz^n$ rather than $z^n$  due to differentiation.
\end{proof}

\subsubsection{Proof of bounds on $T$}

\begin{proof}[Proof of \eqref{e:TNbd}--\eqref{e:TNpbd}]
We are interested in nonnegative $z$ in the disk $U$ so we
take $z\in [0,z_c(1-V^{-1/2})]$.
We reinterpret the bound on $T^{(N)}$ in Lemma~\ref{lem:Delta_as_T_m_S} in terms
of a sum over torus walks.
In this interpretation, we replace the sum over $\Z^d$-walks
\begin{equation}
	\sum_{\omega \in\Wcal_n}\sum_{L \in \Lcal^{(N)}[0,n]}
    \sum_{ab \in L} (-U^+_{ab})\prod_{st\in L\setminus\{ab\}}(-U^\T_{st})
    \prod_{s't' \in \Ccal(L)} (1+\beta U^\T_{s't'})
\end{equation}
from \eqref{e:bd_first_T} by a sum over torus walks
\begin{equation}
\label{e:Twrap}
	\sum_{\omega \in\Wcal_n^\T}\sum_{L \in \Lcal^{(N)}[0,n]}
    \sum_{ab \in L} (-U^{{\rm wrap}}_{ab})\prod_{st\in L\setminus\{ab\}}(-U_{st})
    \prod_{s't' \in \Ccal(L)} (1+\beta U_{s't'}),
\end{equation}
where $U^{{\rm wrap}}_{ab}$ is equal to $-1$ if $\omega(a)=\omega(b)$ via a path
that wraps around the torus and otherwise is equal to zero.
Then walks that make a nonzero contribution to \eqref{e:Twrap}
follow the trajectory of the familiar $N$-loop lace diagram
(from Figure~\ref{fig:pi}) on the torus with the restriction that
at least one of the diagram loops must wrap around the torus as in Figure~\ref{fig:T_long_walk}.

\begin{figure}[t]
\begin{center}
	\includegraphics[scale = 0.5]{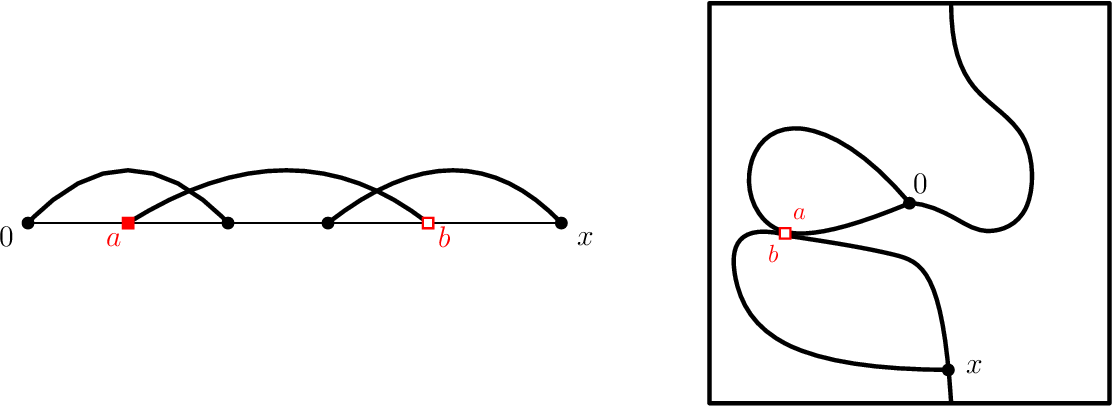}
\end{center}
\caption{ \label{fig:T_long_walk}
An example of a 3-bond lace with a corresponding torus walk configuration
contributing to $T^{(3)}(z)$.
}
\end{figure}

This wrapping loop consists of two (if it is one of the end loops) or three (if it
is an interior loop) subwalks.  One of these subwalks must
have $\ell_\infty$-displacement at least
$r/3$, because the loop travels a distance of at least $r$.
We sum over the two or three cases
which specify which of the subwalks must travel at least $r/3$.  Once that
subwalk is fixed,
using Lemma~\ref{lem:convol_bound}
we bound its diagram line with the supremum norm, and bound all
the other lines as usual by $\|G_z^\T \star G_z^\T\|_\infty^{N-1}$.
The lift of the long subwalk
to $\Z^d$ must have $\ell_\infty$-displacement
at least $r/3$.  It is therefore
bounded above (using $K^{\T} \le K$) by
\begin{align}
\label{e:Twrap1}
    \sup_{x\in\T_r^d} \sum_{u \in \Z^d : \|x+ru\|_\infty \ge r/3} G_z(x+ru)
    .
\end{align}
If the torus point $x$ satisfies $\|x\|_\infty \ge r/3$ then we can simply remove the
restriction on the sum over $u$ and bound the right-hand side using \eqref{e:plateau-ub},
by $G_z(x) + C\chi/V \prec r^{-(d-2)} + \chi/V$.  If instead $\|x\|_\infty < r/3$, then
the $u=0$ term is excluded by the restriction on the summation, and we instead
have a bound just by $\chi/V$ as in \eqref{e:GammachiV}.  Thus, in any case, we obtain a factor
\begin{equation}
    \frac{1}{V}(r^2 + \chi(z))
\end{equation}
from the long line, so as in \eqref{e:PiNbub}
we obtain
\begin{align}
    T^{(N)}(z)
    & \le
    \beta^N O(N)
    \left( \frac{r^2 + \chi(z)}{V} \right)
    \|G_z^\T \star G_z^\T\|_{\infty}^{N-1}
\end{align}
and the desired estimate \eqref{e:TNbd}
follows from \eqref{e:GGstar}.

For the derivative, we can proceed as above but now with
an extra vertex.  If the extra vertex is on the long line
identified as above then we replace
$G_z^\T(x)$ in the above by $(G_z^\T \star G_z^\T)(x)$
(this is as in \eqref{e:nGG}) with the restriction that one of
the two-point functions in the convolution must have a lift whose
displacement is of order $r$.
We bound the supremum norm of this restricted convolution by putting the
$\ell_\infty$-norm on the long factor and the $\ell_1$-norm (which yields $\chi^\T \le \chi$)
on the other.  The $\ell_\infty$ norm is bounded exactly as in the previous
paragraph, so that overall we obtain a bound this line by
\begin{equation}
    \frac{1}{V}(r^2 + \chi) \chi .
\end{equation}

If the extra vertex is not on the long line then we replace the usual weight
$G_z^\T$ for the line containing the extra vertex by $G_z^\T\star G_z^\T$ so that one factor
$\|G_z^\T\star G_z^\T\|_\infty$ gets replaced by $\|G_z^\T \star G_z^\T \star G_z^\T\|_\infty$.
This gives, as in Proposition~\ref{prop:Pidots},
\begin{align}
    T'^{(N)}(z)
    & \le
    \beta^N O(N) \left(
    \frac{r^2\chi(z)
    + \chi(z)^2}{V} \right) \|G_z^\T \star G_z^\T\|_{\infty}^{N-1}
    \nnb & \quad +
    \beta^N O(N) \left( \frac{r^2 + \chi(z)}{V} \right)
    \|G_z^\T \star G_z^\T \star G_z^\T\|_\infty \|G_z^\T \star G_z^\T\|_{\infty}^{N-2}    .
\end{align}
The first term has the upper bound
that we desire.
By \eqref{e:GGGstar} and \eqref{e:GGG}, the three-fold convolution is
less than $\chi +\chi^3/V$,
so overall
the second term contains a factor
\begin{equation}
    \frac{\chi(z)(r^2 + \chi(z))}{V}
    \left(
    1
    + \frac{\chi(z)^2}{V} \right)
    .
\end{equation}
By Remark~\ref{rk:zeta}, this is sufficient.
\end{proof}

\subsubsection{Proof of bounds on $S$}

We now prove the bounds on $S^{(N)}$ and $S'^{(N)}$ stated in \eqref{e:SNbd}--\eqref{e:SNpbd}.
This is the most elaborate part of the proof of Proposition~\ref{prop:AD-bis}.
Recall from \eqref{e:bd_first_S} that
\begin{align}
\label{e:bd_first_S-bis}
	S^{(N)}(z) &\leq
    \beta^{N+1}\sum_{n=1}^\infty z^n \sum_{\omega \in\Wcal_n} \sum_{L \in \Lcal^{(N)}[0,n]}
    \sum_{ab \in \Ccal(L)}(-U^+_{ab})\prod_{st \in L}(-U_{st})
    \prod_{s't' \in \Ccal(L)}(1+\beta U_{s't'})
    .
\end{align}
This contains the desired
factor $\beta^{N+1}$ in $S^{(N)}$ which we
do not carry through the rest of the analysis.

For fixed $\omega$ and $L$,
\begin{align}
\label{e:S_feature}
	\sum_{ab \in \Ccal(L)}(-U^+_{ab})(\omega)
	&=
    \sum_{x \in \Z^d}\sum_{u \neq 0 }\sum_{ab \in \Ccal(L)}
    \1\{\omega(a) = x, \; \omega(b) = x+ru\}.
\end{align}
Thus, the diagram corresponding to $S^{(N)}$ is a standard $\Pi$-diagram (as in Figure~\ref{fig:pi}) with two additional vertices $x$ and $x+ru$ that are
distinct but have the same torus projection.
In the $N$-loop diagram, we label the $2N-1$ subwalks
from $1$ to $2N-1$ in their order of appearance.
From the definition of compatible edges and from
Figure~\ref{fig:S_lace}, we see that the additional two vertices
must belong to two subwalks
whose
labels differ by at most $3$ or are identical.
In what follows, for simplicity we restrict attention to the case where the vertices $a,b$ in
Figure~\ref{fig:S_lace} do not lie on the lines $1,2,2N-2$ or $2N-1$, i.e. on the first or last loop.
The omitted extreme cases are straightforward extensions of our analysis.

We fix
 $z\in [0,z_c(1-V^{-1/2})]$
for the rest of this section.
We will prove the inequalities listed in the following three cases.
Together, they
provide a proof of Proposition \ref{prop:bd_S_T-p} and in fact most of these bounds are better than what is needed.

\smallskip
{\bf \emph{Case~1}:}
$a,b$ occur on the same subwalk: $b=b_1,b'_1$ in Figure~\ref{fig:S_diagram}.
We prove that
\begin{equation}
	\hspace{-0.42cm}\text{ Case~1(a): } S^{(N)}(z)
    \prec
	C^N \frac{r^2}{V}
    ,
    \qquad
	\text{ Case~1(b): }  S'^{(N)}(z)
    \prec
    C^N\frac{ \chi(z)( r^2+ \chi(z))}{V}.
\end{equation}

\smallskip
{\bf \emph{Case~2}:}
$a,b$ occur on the same loop but on different subwalks: $b=b_2,b_3,b'_2,b'_3,b'_4,b'_5$ in Figure~\ref{fig:S_diagram}. We prove that
\begin{equation}
	\text{ Case~2(a): } S^{(N)}(z)  \prec
    C^{N}\frac{\chi(z)}{V}\frac{1}{r^{d-4}}, \qquad
	\text{ Case~2(b): }S'^{(N)}(z) \prec
    C^{N}\frac{\chi(z)^2}{V}\frac{1}{r^{(d-4)/2}}.
\end{equation}

\smallskip
{\bf \emph{Case~3}:}
$a,b$ occur on adjacent loops: $b=b_4,b_5,b_6,b'_6$ in Figure~\ref{fig:S_diagram}. We prove that
\begin{equation}
	\text{ Case~3(a): }S^{(N)}(z) \prec
    C^N \frac{\chi(z)}{V}\frac{1}{r^{d-4}}, \qquad
	\text{ Case~3(b): }S'^{(N)}(z) \prec
    C^N \frac{\chi(z)}{V}\frac{1}{r^{(d-4)/2}}.
\end{equation}

\begin{figure}[h]
\centering
    \includegraphics[width=14cm, height=3cm]{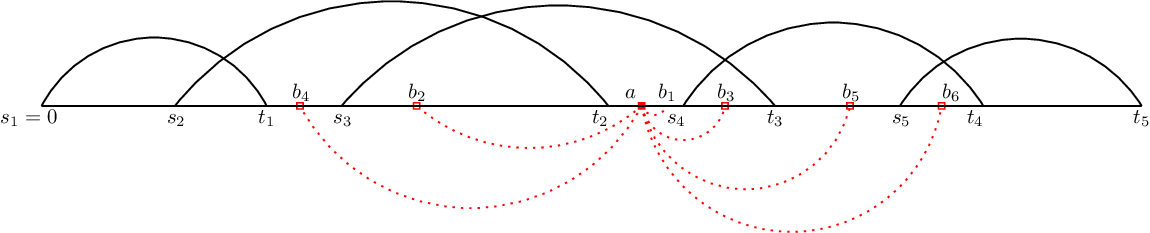}
	\caption{A $5$-edge lace
with  compatible edges for $a$
depicted with dotted lines: there are $6$ subwalks
where $b$ may occur, with
generic locations denoted $b_1,\ldots,b_6$.}
	\label{fig:S_lace}
\end{figure}

\begin{figure}[h]
\centering
    \includegraphics[scale = 0.75]{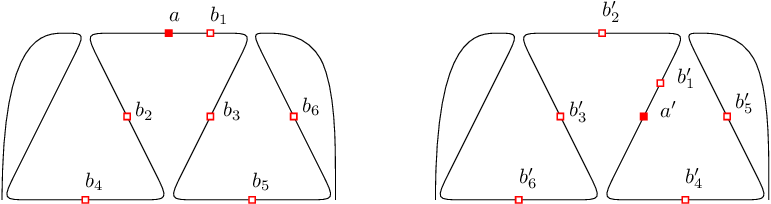}
	\caption{Walk diagrams for the lace in Figure~\ref{fig:S_lace} with possible locations for the vertices $x$ and $x+ru$ from \eqref{e:S_feature}
represented by squares and boxes. Two possible configurations are depicted.}
	\label{fig:S_diagram}
\end{figure}

\bigskip \noindent {\bf \emph{Case~1(a) for $S^{(N)}$.}}
In this case,
using Lemma~\ref{lem:convol_bound}
 we bound the line containing $a,b$ with the supremum norm and bound all
other lines with the bubble.  We obtain $C^{N-1}$ times
\begin{align}
\label{e:case_1_first_bd}
    \sup_{x\in\Z^d} \sum_{v\in\Z^d} \sum_{u \neq 0} G_z(v) G_z(ru)
    G_z(x - v -ru),
\end{align}
and we now focus on bounding \eqref{e:case_1_first_bd}.
Since $u \neq 0$ we have $G_z(ru) \prec r^{-(d-2)}$.
This gives an upper bound, up to multiplicative constant, of the form
\begin{align}
	r^{-(d-2)}\sup_{x\in\Z^d}\sum_{v\in\Z^d} \sum_{u \neq 0}G_z(v)   G_z(x- v -ru)
    =
    r^{-(d-2)}\sup_{x\in\Z^d}\sum_{u \neq 0} \bubble_z(x-ru).
\end{align}
By \eqref{e:Bbd} and the fact that $\bubble(y)$ is maximal at $y=0$
by the Cauchy--Schwarz inequality, this
is bounded by
\begin{equation}
    \frac{1}{r^{d-2}} \Big(\bubble_{z}(0) + C\frac{\chi(z)}{V} \Big),
\end{equation}
from which we see that
\begin{align}
	 \text{Case~1(a)}
    &\prec C^N\frac{r^2}{V}\Big(1+\frac{\chi(z)^2}{V}\Big).
\end{align}
By Remark~\ref{rk:zeta}, this is sufficient.
\qed

\bigskip \noindent {\bf \emph{Case~1(b) for $S'^{(N)}$.}}
If the extra vertex from the derivative is on a different line than the one containing $a,b$,
then we replace one factor $\|G_z*G_z\|_\infty$ by $\|G_z*G_z*G_z\|_\infty$,
which by \eqref{e:GGG} is at most $\chi(z)$.  This increases the bound from
Case~1(a) by a factor $\chi$, leading to an overall bound from this contribution of the form
\begin{equation}
	C^{N}\frac{r^2\chi(z)}{V}\Big(1+\frac{\chi(z)^2}{V}\Big).
\end{equation}
If the extra vertex is instead on the same line as $a,b$ then we bound that line
as in the bound on $\widetilde \psi$ (recall Figure~\ref{fig:one_loop_prime})
in the proof of
Proposition~\ref{prop:Delta_1_prime_bound-upper} with
a multiple of the supremum over $x\in\Z^d$ of
\begin{equation}
\label{e:1prime}
    \sum_{u\neq 0}\Big(\bubble_z(x-ru)\bubble_z(ru)+ {\sf T}_z(x-ru)G_z(ru)\Big).
\end{equation}
To see in detail how one of these terms arises, the first diagram of
Figure~\ref{fig:one_loop_prime} with the origin moved to the
box represents
\begin{equation}
    \sum_{u \neq 0} G_z(ru) \sum_{v,w\in\Z^d} G_z(v)G_z(w-v)G_z(w+x-ru)
    =
    \sum_{u \neq 0} G_z(ru) {\sf T}_z(x-ru).
\end{equation}
The sum \eqref{e:1prime} is shown in Lemma~\ref{lem:S_prime_case_1_as_psi}
to be bounded by $V^{-1}\chi(z)^2$ uniformly over $x \in \Z^d$.
Finally, summing both cases gives
\begin{equation}
	\text{ Case~1(b)}
    \prec C^N\frac{\chi(z)^2}{V} + C^{N}\frac{r^2\chi(z)}{V}\Big(1+\frac{\chi(z)^2}{V}\Big).
\end{equation}
This is sufficient by Remark~\ref{rk:zeta}.
\qed

\medskip
For Cases~2 and 3, we isolate some useful estimates in Lemma~\ref{lem:psi_psi_tilde_bds}.
For this, we define functions $\Psi_z$ and $\widetilde \Psi_z$ on $\Z^d$ by
\begin{align}
\label{e:psi_def}
	\Psi_z(x)
	&= \Big[\sum_{u \neq 0}(G_z^2*G_z^2)(x+ru)\Big]^{1/2},\\
\label{e:psi_tilde_def}
	\widetilde \Psi_z(x)
    &= \Big[\sum_{u \neq 0}\left( (G_z*G_z)^2*G_z^2\right) (x+ru)\Big]^{1/2}.
\end{align}

\begin{lemma}
\label{lem:psi_psi_tilde_bds}
Let $d>4$ and let $\beta>0$ be sufficiently small.
For any $z<z_c$,
\begin{align}
	\Psi_z(0)
	&\prec \frac{1}{r^{d-2}}, \qquad
    \|\Psi_zG_z\|_2
	\prec \frac{1}{r^{d-2}},  \\
	\widetilde \Psi_z(0)
	&\prec \frac{\chi(z)}{V^{1/2}}, \qquad
    \|\widetilde \Psi_z G_z\|_2
	\prec \frac{\chi(z)}{V^{1/2}},
\end{align}
and these bounds remain satisfied with any $G_z$ (including those in $\Psi_z$
and $\widetilde\Psi_z$) replaced by $(G_z^2*G_z^2)^{1/2}$.
\end{lemma}

\begin{proof}
By the first case of
Lemma~\ref{lem:conv-nu}, $(G_z^2*G_z^2)(x) \prec \xvee^{-(2d-4)}e^{-2 c_1m(z)\|x\|_\infty}$, so by
Lemma~\ref{lem:unifmassint} with $a=4-d<0$ and $x=0\in\T_r^d$
we obtain the desired bound $\Psi_z(0)^2 \prec
r^{-(2d-4)}$.
Similarly, $[(G_z*G_z)^2*G_z^2](x) \prec \xvee^{-(2d-8)} e^{- 2c_1m(z)\|x\|_\infty}$
and from Lemma~\ref{lem:unifmassint} with $a=8-d$ we obtain the desired estimate
\begin{equation}
\label{e:2d-8_bd}
	\widetilde \Psi_z(0)^2
    \prec
    \sum_{u\neq 0}\frac{e^{- 2c_1 m(z)\|ru\|_\infty}}{\veee{ru}^{2d-8}}
    \prec
    \frac{\chi(z)^2}{V},
\end{equation}
where we have chosen the last bound so as to
accommodate the worst case which is
$d \downarrow 4$.

For the $\ell_2$ bounds, we first use the definition of $\Psi_z$ to see that
\begin{align}
\label{e:PsiG}
    \|\Psi_z G_z\|_2^2
    & =
    \sum_{u \neq 0} \sum_x G_z(x)^2 (G_z^2*G_z^2)(x+ru)
    =
    \sum_{u \neq 0} (G_z^2*G_z^2*G_z^2)(ru).
\end{align}
By Lemma~\ref{lem:conv-nu}, $(G_z^2*G_z^2*G_z^2)$ obeys the same estimate as $G_z^2*G_z^2$
used in the previous paragraph to bound $\Psi_z(0)$,
so \eqref{e:PsiG} is also bounded by $r^{-(2d-4)}$, as required.
Similarly,
\begin{align}
\label{e:PsitilG}
    \|\widetilde\Psi_z G_z\|_2^2
    & =
    \sum_{u \neq 0} ((G_z*G_z)^2*G_z^2*G_z^2)(ru),
\end{align}
and the extra factor $G_z^2$ in the convolution, compared to the bound on $\widetilde\Psi_z(0)$, again has no effect and we again obtain an upper bound $V^{-1}\chi(z)^2$ as in \eqref{e:2d-8_bd}.

Finally, by Lemma~\ref{lem:conv-nu} and since $2d-4 >d$,
\begin{equation}
	(G_z^2*G_z^2)^{1/2}(x) \prec \frac{e^{-c_1m(z)\|x\|_\infty}}{\xvee^{d-2}},
\end{equation}
which is the same as the bound on $G_z$ that we have used throughout this proof.
Therefore the above discussion also applies when we replace any $G_z$ by $(G_z^2*G_z^2)^{1/2}$.
\end{proof}

\begin{figure}[t]
\centering
	\includegraphics[scale = 0.6]{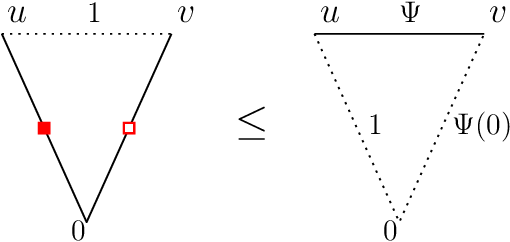}
	\caption{Diagrammatic bound for Case~2(a). Lines other than $G_z$ lines have their associated weights written next to them. Lines with constant weights are dotted.}
	\label{fig:S_case_2}
\end{figure}

\bigskip \noindent {\bf \emph{Case~2(a) for $S^{(N)}$.}}
We place the origin at the single common endpoint of the two distinct subwalks containing $a$ and $b$, as in Figure~\ref{fig:S_case_2}.
The portion of the diagram corresponding to these two lines is
\begin{equation}
	\sum_{x\in \Z^d}\sum_{s\neq 0} G_z(x)G_z(u-x)G_{z}(x+rs)G_z(v-x-rs)
    .
\end{equation}
We reorganise the above sum and apply the Cauchy--Schwarz inequality to see that it is bounded
by
\begin{align}
\label{e:S_2_lines_bd}
    & \sum_{s \neq 0} \sum_{x\in \Z^d}
    \left( G_z(x)G_z(x+rs) \right) \left( G_z(u-x)G_z(v-x-rs) \right)
    \nnb & \quad \le
    \Big( \sum_{s \neq 0} (G_z^2*G_z^2)(rs) \Big)^{1/2}
    \Big( \sum_{s \neq 0} (G_z^2*G_z^2)(v-u+ rs) \Big)^{1/2} = \Psi_z(0) \Psi_z(v-u).
\end{align}
We interpret this inequality as an upper bound in which the weights of the
original diagram lines containing $a$ and $b$
are replaced by constant functions $1$ and $\Psi_z(0)$, and
the weight of the line joining $u$ and $v$
is replaced by $\Psi_z G_z$ (instead of the original $G_z$).
This is depicted in Figure~\ref{fig:S_case_2}.

We bound this new diagram using Lemma~\ref{lem:convol_bound}.  We place
the supremum norm on the line containing $\Psi_z(0)$.
There is one factor with lines $G_z$ and $1$ which is dominated by $\|G_z*1\|_{\infty}$.
Another factor with lines $G_z$ and $\Psi_z G_z$ is controlled by $\|(\Psi_z G_z)*G_z\|_{\infty}$
All the other factors have $G_z$ lines and
give standard factors involving $G_z*G_z$ as usual.
By combining the above contributions, for $N \ge 3$ we obtain the upper bound
\begin{equation}
	\|G_z*G_z\|_\infty^{N-3}\|G_z*1\|_\infty \|(\Psi_z G_z) * G_z\|_\infty \|\Psi_z(0)\|_\infty.
\end{equation}
By using Cauchy--Schwarz for the first and third factors and rewriting the second and fourth ones we can simplify the bound as
\begin{equation}
\label{e:bubble_extr_S_2}
	\|G_z\|_2^{2N-5}\,\|G_z\|_1\,\|\Psi_zG_z\|_2\,\Psi_z(0).
\end{equation}
Since $\|G_z\|_1 = \chi(z)$ and $\|G_z\|_2 \leq \|G_{z_c}\|_2 < \infty$, it follows from Lemma~\ref{lem:psi_psi_tilde_bds} that \eqref{e:bubble_extr_S_2} is bounded above by
\begin{equation}
\label{e:bd_S_case_2}
	\text{ Case~2(a)}
    \prec C^{N}\chi(z)\frac{1}{r^{d-2}} \frac{1}{r^{d-2}}
    = C^N \frac{\chi(z)}{V}\frac{1}{r^{d-4}}.
\end{equation}
This gives the claimed result.
\qed

\bigskip \noindent {\bf \emph{Case~2(b) for $S'^{(N)}$.}}
The derivative adds a vertex to the diagrams arising in Case~2(a).  This
amounts to replacing exactly one of the weights $G_z$ by $G_z*G_z$.
We thus obtain an upper bound by adapting
the proof of Case~2(a) by replacing exactly one of the $G_z$ factors in \eqref{e:bubble_extr_S_2} (including the ones in $\Psi_z$) by $G_z*G_z$.
Note that $\widetilde \Psi_z$ is exactly obtained from $\Psi_z$ by replacing one of its $G_z$ factors by $G_z*G_z$.
This leads to one of the new factors, compared to \eqref{e:bubble_extr_S_2},
\begin{align}
\label{e:new_factors_S_2_G}
	\|G_z*G_z\|_2 &\leq \chi(z) ,
\qquad
	\|G_z*G_z\|_1 = \chi(z)^2, \\
\label{e:new_factors_S_2_Psi}
	\|\widetilde \Psi_zG_z\|_2 = \|\Psi_z(G_z*G_z)\|_2 &\leq \frac{\chi(z)}{V^{1/2}}, \qquad
	 \widetilde \Psi_z(0) \leq \frac{\chi(z)}{V^{1/2}}.
\end{align}
The first inequality follows as usual from Lemma~\ref{lem:unifmassint}
and is sharp when $d \downarrow 4$, the second is an elementary identity,
and the third and fourth follow from Lemma~\ref{lem:psi_psi_tilde_bds} together with the
identity $\|\widetilde \Psi_zG_z\|_2 = \|\Psi_z(G_z*G_z)\|_2$
(which holds by definition).
The net effect on \eqref{e:bd_S_case_2} is to replace a factor $r^{-(d-2)}$ by
$V^{-1/2}\chi$ or to multiply by an additional factor $\chi$.  Since $r^{d-2}V^{-1/2} >1$,
the former replacement dominates and we conclude that
\begin{equation}
    \text{ Case~2(b)}
    \prec C^{N}\chi(z)\frac{1}{r^{d-2}} \frac{\chi(z)}{r^{d/2}}
    = C^N \frac{\chi(z)^2}{V}\frac{1}{r^{(d-4)/2}}.
\end{equation}
\qed

\bigskip \noindent {\bf \emph{Case~3(a) for $S^{(N)}$.}}
Consider first the case where the lines containing $a$ and $b$ have no common endpoint (which can occur
only for $N\geq 4$), as in Figure~\ref{fig:S_case_3_1}.
These two lines involve four $G_z$ factors and by the Cauchy--Schwarz inequality
their contribution obeys
\begin{align}
	\sum_{x\in \Z^d} \sum_{s \neq 0}
    (G_z(v-x)G_z(w-x-rs))(G_z(x-u)G_{z}(x+rs))
	&\leq \Psi_z(w-v)\Psi_z(u).
\end{align}
This amounts to replacing each of the two opposite lines of a pair of consecutive loops
by $\Psi_z G_z$ instead of $G_z$, and the other pair of lines by $1$, as in
Figure~\ref{fig:S_case_3_1}.
We bound this new diagram using Lemma~\ref{lem:convol_bound}:
we place the supremum norm on one of the $1$ lines and extract the rest of the line pairs.
Two pairs have $\Psi_zG_z$ and $G_z$ instead of a single pair
as in Case~2(a), and there is also one pair with $G_z$ and $1$. The rest are standard pairs with $G_z$ lines. This gives an overall upper bound
\begin{align}
	\|G_z*G_z\|_\infty^{N-4} \|G_z*1\|_\infty \|(\Psi_z G_z)*G_z\|_\infty^2 \|1\|_\infty.
\end{align}
By
Lemma~\ref{lem:psi_psi_tilde_bds}, we obtain an upper bound on this case by
\begin{equation}
	\|G_z\|_2^{2N-6}\|G_z\|_1 \|\Psi_zG_z\|_2^2
	\leq C^N  \chi(z)\frac{1}{r^{2(d-2)}} = C^N\frac{\chi(z)}{V}\frac{1}{r^{d-4}}.
\end{equation}

\begin{figure}[h]
\centering
	\includegraphics[scale = 0.6]{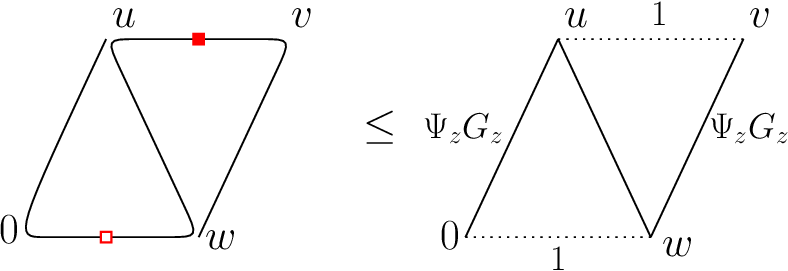}
	\caption{Schematic representation of the bound used for the first case of Case~3(a).
    Lines with constant weights are dotted.
    }
	\label{fig:S_case_3_1}
\end{figure}

\begin{figure}[h]
\centering
	\includegraphics[scale = 0.6]{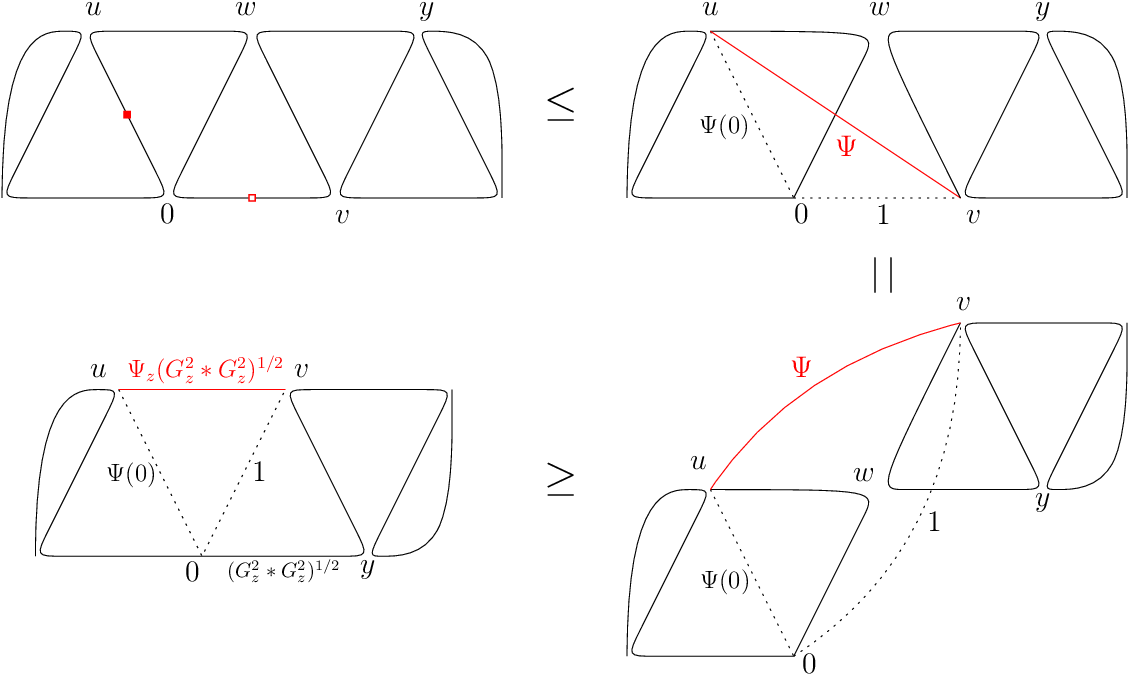}
	\caption{Schematic representation of the bound used for the second case of Case~3(a).
We use the Cauchy--Schwarz inequality twice.
The line between $u$ and $v$ is first non-existent (or equivalently constant equal to $1$) then $\Psi$ and finally $\Psi (G^2*G^2)^{1/2}$. Lines with constant weights are dotted.
}
	\label{fig:S_distance_3_bound}
\end{figure}

If instead the two lines containing $a$ and $b$ share an endpoint, then we
place the origin of the diagram at this vertex and apply the Cauchy--Schwarz inequality exactly as in \eqref{e:S_2_lines_bd}. This does not give a standard diagram,
due to the atypical diagonal $\Psi$ line in Figure~\ref{fig:S_distance_3_bound}.
To deal with this, we
apply the Cauchy--Schwarz inequality a second time,
to the portion of the diagram with this diagonal line where the sum runs over the one vertex in the two loops that does not belong to either of the two lines ($w$ in Figure~\ref{fig:S_distance_3_bound}).
This gives
\begin{align}
\label{e:2nd_CS_case_3}
	\sum_{w} [G_z(w)G_z(y-w)][G_z(v-w)G_z(u-w)]
	&\leq (G_z^2*G_z^2)^{1/2}(v-u) (G_z^2*G_z^2)^{1/2}(y) .
\end{align}
Diagrammatically, as illustrated in Figure~\ref{fig:S_distance_3_bound},
this has the effect of removing a loop to produce the diagram for
Case~2(a) for $S^{(N-1)}$ (with slightly modified weights) instead of $S^{(N)}$.
Since our estimates apply also when $G_z$ is replaced by $(G_z^2*G_z^2)^{1/2}$
(recall Lemma~\ref{lem:psi_psi_tilde_bds}), we
therefore again obtain a bound
$C^N V^{-1} \chi(z) r^{-(d-4)}$
as in \eqref{e:bd_S_case_2}.
Both contributions in Case~3(a) are equal and the overall bound is thus
\begin{equation}
	\text{ Case~3(a)}
    \prec  C^{N}\frac{\chi(z)}{V}\frac{1}{r^{d-4}},
\end{equation}
which
is sufficient.
\qed

\bigskip \noindent {\bf \emph{Case~3(b) for $S'^{(N)}$.}}
As in Case~2(b), the derivative
effectively replaces
exactly one of the weights $G_z$ encountered in Case~3(a) by $G_z*G_z$.
We use the estimates obtained in \eqref{e:new_factors_S_2_G}-\eqref{e:new_factors_S_2_Psi}.
The only difference comes from the second application of the Cauchy--Schwarz inequality in \eqref{e:2nd_CS_case_3} which changes some $G_z$ factors to $(G^2_z*G^2_z)^{1/2}$. However, since both functions are bounded in the same way
(as in Lemma~\ref{lem:psi_psi_tilde_bds})
this leads again to an overall bound
\begin{equation}
\text{ Case~3(b)}
\prec  C^{N}\frac{\chi(z)^2}{V}\frac{1}{r^{(d-4)/2}}.
\end{equation}
This completes the proof.
\qed

\section*{Acknowledgements}

This work was supported in part by NSERC of Canada.
GS thanks David Brydges and Tyler Helmuth for
helpful discussions during the early stages of this work.
We thank Tom Hutchcroft for comments on a preliminary version of the paper.

%\bibliographystyle{plain}
%\bibliography{cn}

\end{document}